\documentclass[a4paper,12pt]{article}
\usepackage{complexity}
\usepackage{lineno}
\usepackage{xcolor}
\usepackage[T1]{fontenc}
\usepackage{graphicx}
\usepackage{textgreek}
\usepackage{pstricks,pspicture}
\usepackage{float}
\usepackage{amsthm,amsmath,amssymb}
\usepackage{mathrsfs,amssymb,amsmath,amsthm,amstext,amsfonts}
\usepackage{pst-plot}
\usepackage{epsfig}
\usepackage{setspace}
\usepackage{float}
\usepackage{rotating}
\usepackage{tikz}
\usepackage{enumitem}
\usepackage{multicol} 
\usepackage[left=2cm,top=2cm, bottom=2cm,right=2cm]{geometry}

\DeclareMathOperator{\car}{cth_{cc}}
\DeclareMathOperator{\nt}{nt}
\DeclareMathOperator{\trivial}{t}
\newcommand{\intv}{I_{cc}}
\DeclareMathOperator{\hullc}{Hull_{cc}}

\newtheorem{theorem}{Theorem}

\newtheorem{lemma}[theorem]{Lemma}
\newtheorem{corollary}[theorem]{Corollary}
\newtheorem{proposition}[theorem]{Proposition}

\newtheorem{problem}[theorem]{Problem}

\title{Carath\'eodory Number in Cycle Convexity}
\author{Revathy S. Nair$^{a}$\footnote{revathyrahulnivi@gmail.com}
\and Bijo S. Anand$^{b}$\footnote{bijos\_anand@yahoo.com}\and Ullas Chandran S. V.$^{c}$ \footnote{svuc.math@gmail.com}  \and Julliano R. Nascimento$^{d}$\footnote{jullianonascimento@ufg.br}
 \\\\
 $^{a}$ \small Department of Mathematics, Mar Ivanios College, University of Kerala, \\\small Thiruvananthapuram, India\\
$^{b}$\small Department of Mathematics, Sree Narayana College, Punalur, Kerala \\$^{c}$\small Department of Mathematics, Mahatma Gandhi College, \\\small Thiruvananthapuram - 695004, Kerala, India\small \\$^{d}$\small Instituto de Informática, Universidade Federal de Goiás, Goiânia, GO, Brazil
}
\date{\today}

\begin{document}

\maketitle

\begin{abstract}
Let $G$ be a graph and $S \subseteq V(G)$. In the cycle convexity, we say that $S$ is \textit{cycle convex} if for any $u\in V(G)\setminus S$, the induced subgraph of $S\cup\{u\}$ contains no cycle that includes $u$. The \textit{cycle convex hull} of $S$, denoted by $\hullc (S)$, is the smallest cycle convex set containing $S$. 
A set $S \subseteq V(G)$ is said to be \textit{Carathéodory independent} if there exists a vertex $u \in \hullc(S) $ such that $u \notin\displaystyle \bigcup_{a \in S} \hullc (S \setminus \{a\}) $, and the Carathéodory number $\car(G)$ is the maximum size of such a set.
In this paper, we prove that given a graph $G$ and $k \in \mathbb{N}$, deciding whether $\car(G) \geq k$ is \NP-complete, even when $G$ is bipartite. On the other hand, we derive exact values and constant upper bounds for several graph classes, leading to polynomial-time algorithms. Some of them include forests, cycles, complete graphs, complete multipartite, split, and $P_4$-sparse graphs. In addition, we present a characterization of $n$-vertex graphs $G$ with extremal values near to $n$, including $\car(G) = n-1$ and $\car(G) = n-2$. Furthermore, we investigate the behavior of the Carathéodory number under graph products such as the strong, lexicographic, and Cartesian products.
\end{abstract}

\noindent{\small {\bf AMS Subj.Class:} 05C69, 05C76, 05C85}

\section{Introduction}

Convexity spaces form a rich research topic, providing an abstract treatment to notions that go beyond convex sets in Euclidean geometry. Applied to a graph $G$, a set ${\cal C}$ of subsets of $V(G)$ is a \textit{convexity} in $G$ if $(i)$ $\emptyset, V(G) \in {\cal C}$ and $(ii)$ ${\cal C}$ is closed under intersection. Each element of ${\cal C}$ is called a \textit{convex set}. The \textit{convex hull}, $\hullc (S) $ of a subset $S \subseteq V(G)$ is the smallest convex set containing $S$. For more details, we refer to~\cite{duchet1987convexity, van22}. 

A family of paths $\mathcal{P}$ in a graph $G$ give rise to the most natural graph convexities. A set $S \subseteq V(G)$ is $\mathcal{P}$-convex if it contains all vertices of every path in $\mathcal{P}$ between any two vertices of $S$. Classical examples include \textit{geodesic convexity}, where $\mathcal{P}$ is the set of shortest paths~\cite{everett1985hull, buckley-1990, farber-1986}, as well as \textit{monophonic convexity}~\cite{caceres-2005, source17, duchet1987convexity} and $P_3$-convexity~\cite{source11, centeno2011irreversible, coelho2019p3}, defined respectively over induced paths and paths on three vertices.

There are also convexity notions not defined via paths. Some examples are \textit{Steiner convexity}~\cite{source9} and $\Delta$\textit{-convexity}~\cite{bijo2,bijo3,bijo1}. In Steiner convexity, a set $S$ is st\textit{-convex} if it contains all vertices of any Steiner tree whose terminal set is contained in $S$. In $\Delta$-convexity, a set $S$ is $\Delta$\textit{-convex} if no vertex outside $S$ forms a $C_3$ with two vertices of $S$.

Whether defined via paths or not, convexity in graphs are useful to capture spreading processes among the entities. This perspective relates to graph-theoretical models for the diffusion of information, opinions, or diseases, such as those studied in~\cite{dreyer2009irreversible}. These connections have also motivated the study of algorithmic aspects of graph convexities; ; see the book~\cite{araujo2025introduction}.

Observe that $\Delta$-convexity is based on cycles on three vertices. One can ask for notions defined using larger cycles. In 2018, Araújo et al.~\cite{interval08} introduced \textit{cycle convexity}, that find applications in Knot Theory~\cite{araujo2020cycle} and provides such a generalization.

For a set $S$ of vertices in a graph $G$, the \textit{cycle interval} of $S$, denoted by $\intv(S)$, is the set formed by the vertices of $S$ and any $w \in V(G)$ that form a cycle with some vertices of $S$. We say that a vertex $w \in V(G)\setminus S $ is \textit{generated} by $S$ if $w\in \intv(S)$. If $\intv(S) = S$, then $S$ is \textit{cycle convex} in $G$. The \textit{cycle convex hull} of a set $S$, denoted by $\hullc (S) $, is the smallest cycle convex set containing $S$. The cycle convex hull of $S$ can be obtained by successive applications of the cycle interval operation. For that, we define $\intv^0(S) = S$, $\intv^1(S) = \intv(S)$, and $\intv^k(S) = \intv(\intv^{k-1}(S))$, for every $k \geq 2$. With this terminology, $\hullc (S) = \intv^k(S)$, for any $k \geq 1$ such that the equality $\intv^k(S) = \intv^{k-1}(S)$ holds.

Several parameters have been investigated in cycle convexity. Classical examples include the interval~\cite{interval08} and hull numbers~\cite{hull09}, as well as the convexity number~\cite{anand2025complexity,lima2024complexity}. Other studies consider parameters such as the rank and general position numbers~\cite{araujo2025on}, 
convex partitions~\cite{gomes2025some}, and percolation time~\cite{lima2024complexity}. In this work, we focus on the Carath\'{e}odory number.

\smallskip
Let $G$ be a graph and $S \subseteq V(G)$. The set $S$ is \textit{Carathéodory dependent (or, C-dependent)} provided $\hullc (S) \subseteq \mathop{\scalebox{1.1}{$\bigcup$}}\limits_{a \in S}\hullc (S \setminus \{a\})$ and it is \textit{Carathéodory independent (or, C-independent)} otherwise. That is, $S$ is said to be $\textit{C-independent}$, if there is a $p \in \hullc (S)  \setminus \mathop{\scalebox{1.1}{$\bigcup$}}\limits_{a \in S}\hullc (S \setminus \{a\})$. The \textit{Carath\'{e}odory number} of $G$, denoted by $\car(G)$ is the maximum cardinality of a $\textit{C-independent}$ set of $G$. 

\smallskip
The Carathéodory number has also been studied in other graph convexities. In $P_3$- and geodesic convexities, deciding whether a given graph $G$ and an integer $k$ has Carathéodory number at least $k$ is \NP-complete~\cite{barbosa2012caratheodory,dourado2013caratheodory}. On the other hand, polynomial-time results are known for some restricted graph classes, such as chordal graphs in $P_3$-convexity~\cite{coelho2014caratheodory} and split graphs in geodesic convexity~\cite{dourado2013caratheodory}. In monophonic convexity, the Carathéodory number is at most $2$ for any graph~\cite{source20}. More recently, in $\Delta$-convexity, this parameter has been studied on several graph classes and graph products~\cite{anand2025caratheodory}.

In this paper, we prove that, given a graph $G$ and a positive integer $k$, deciding whether $\car(G) \geq k$ is \NP-complete, even when $G$ is bipartite. On the positive side, we derive constant upper bounds for several graph classes, including forests, cycles, complete, complete multipartite, split, and $P_4$-sparse graphs. We also characterize families of graphs with Carathéodory number equal to $n-1$ which comprise of the cycles $C_n$ and we introduce a family with value $n-2$. All those results lead to polynomial-time algorithms. Further, we establish results on graph products, including the strong, lexicographic, and Cartesian products.

The paper is organized as follows. Section~\ref{sec:pre} presents some additional terminology, basic definitions and preliminary results. Section~\ref{sec:complexity} contains the hardness results. In Section~\ref{sec:extremal}, we present some results for specific graph classes. Section~\ref{sec:products} contains results on graph products and Section~\ref{sec:conclusions} presents some final remarks.

\section{Preliminaries}
\label{sec:pre}

All the graphs considered in this paper are connected, simple, and undirected. 
We denote the graph by $G=(V,E)$ and when needed, we write $V=V(G)$ and $E=E(G)$ to avoid ambiguity.
Given a graph $G$ and $u \in V(G)$, the \index{open neighbourhood} open and the  \index{closed neighbourhood} closed neighbourhood of a vertex $u$ are $N_{G}(u) = \{v : uv \in E(G)\}$ and $N_{G}[u] = N_{G}(v) \cup \{u\}$, respectively. 
The \textit{degree} of a vertex $v \in V(G)$ is $\deg_G(v) = |N_G(u)|$. We omit $G$ when such a graph is clear from the context. Two vertices $u,v$ in a graph $G$ are called \textit{true twins} if $N_G[u]=N_G[v]$ and \textit{false twins} if $N_G(u)=N_G(v)$. A graph $G$ is \textit{connected} if there exists a path between every pair of vertices in $G$. Otherwise, $G$ is \textit{disconnected}. A \textit{component} of a graph $G$ is a maximal connected subgraph of $G$. A vertex $v$ in a connected graph $G$ is a \textit{cut-vertex} of $G$ if $G - v$ is disconnected. A connected graph $G$ is said to be $\textit{2-connected}$, if $G\setminus\{v\}$ is connected for every vertex $v \in V(G)$. Given a graph $G=(V,E)$ and a vertex set $S \subseteq V(G)$,
we denote by $G[S]$ the \emph{subgraph of $G$ induced by $S$}, that is, the graph with vertex set $S$ and edge set
$E(G[S]) = \{ uv \in E(G) : u,v \in S \}$.

We denote by $P_n$, $C_n$ and $K_n$, a path, cycle, and complete graph on $n$ vertices, respectively. A graph $G$ is \textit{unicyclic}, if it is connected and contains precisely one cycle. A \textit{bicyclic graph} is a simple connected graph in which the number of edges equals the number of vertices plus one. A vertex $v \in V(G)$ is a $\textit{universal vertex}$ of $G$, if $v$ is adjacent to all other vertices of $G$. 

For a positive integer $i$, we denote $[i]$ the set $\{1, 2,\ldots, i\}$. An \textit{independent set} (\textit{clique}) of $G$ is a subset $S \subseteq V(G)$ such that its vertices are pairwise non-adjacent (adjacent) in $G$. A \textit{complete $q$-partite (multipartite) graph} is a graph that admits a partition of its vertex set into $q$ independent sets $V_1, V_2,\dots, V_q$, where every vertex in $V_i$ is adjacent to all vertices in $V_j$, for every $j \neq i$, where $i,j \in [q]$. Two graphs $G=(V ,E)$ and $G'=(V',E')$ are \textit{isomorphic} if and only if there is a bijection, called \textit{isomorphism function}, $\varphi: V(G) \to V'(G')$ such that $uv \in E(G)$ if and only if $\varphi(u)\varphi(v) \in E'(G')$, for every $u,v \in V(G)$. We denote by $G \simeq G'$ if $G$ and $G'$ are isomorphic. Let $G$ and $H$ be two graphs. We say that $G$ is an $H$\textit{-free graph} if $G$ does not contain any subgraph isomorphic to $H$. A \textit{cograph} is a $P_4$-free graph. Given two disjoint graphs $G_1$ and $G_2$, the \textit{disjoint union} $G_1 \cup G_2$ is the graph $(V(G_1)\cup V(G_2),\, E(G_1)\cup E(G_2))$. The \textit{join} $G_1 + G_2$ is obtained from $G_1 \cup G_2$ by adding all possible edges between the vertices in $V(G_1)$ and $V(G_2)$.

We begin with basic properties of Carathéodory independent sets. The following lemma includes known results from Anand et al.~\cite{anand2025caratheodory} and Barbosa et al.~\cite{barbosa2012caratheodory}, that also holds for cycle convexity, together with some additional properties.

\begin{lemma}\label{lemma:cth_ind_set_cc}
If $S$ is a $C$-independent set of a graph $G$ in cycle convexity and $|S|\geq 2$, then:
\begin{enumerate}[label={\normalfont(\alph*)}]     
    \item $S$ induces a graph containing at least one edge~\textnormal{\cite{anand2025caratheodory}}.    
    \item every vertex in $S$ is part of some cycle in $G$~\textnormal{\cite{anand2025caratheodory}}.    
    \item $\hullc (S) $ induces a connected subgraph of $G$~\textnormal{\cite{barbosa2012caratheodory}}.
    \item no proper subset $S'$ of $S$ satisfies $\hullc (S') = V(G)$~\textnormal{\cite{barbosa2012caratheodory}}.
    \item $S$ induces a forest.\label{C-independent ppty}
    \item $\hullc (S)$ induces a graph with no leaves.
    \item $S$ is not convex.    
\end{enumerate}
\end{lemma}

\begin{proof}
Items~(a) and~(b) follow from Anand et al.~\cite{anand2025caratheodory} and Items~(c) and~(d) follow from Barbosa et al.~\cite{barbosa2012caratheodory}, since the same reasoning applies to cycle convexity. 
    \begin{enumerate}[label={\normalfont(\alph*)}, start=5] 
    \item By contradiction, suppose that $G[S]$ contains a cycle $C$. Then, for any $u\in V(C)$, we have $u\in \hullc (V(C)\setminus \{u\})$, consequently $\hullc (S \setminus \{u\}) = \hullc (S)$. Therefore, $\hullc (S) \subseteq \displaystyle\bigcup_{a\in S}\hullc (S\setminus\{a\})$ and $S$ is a C-dependent set, a contradiction. 

    \item Suppose, by contradiction that $u$ is a leaf in $G[\langle S \rangle]$. Let $N_{G[\hullc (S)]}(u) = \{ u'\}$. Since $S$ is a Carathéodory set, there exists $v \in \hullc (S) \setminus\displaystyle \bigcup_{a \in S} \hullc (S \setminus \{a\})$. Given that $u$ is a leaf, we know that $u \in S$. By Item~(c), $G[\hullc (S)]$ is connected, then $u' \in \hullc (S)$. However, since $u'$ is a cut-vertex, the membership of $u'$ in $\hullc (S)$ is not due to $u$. Hence $v \in \hullc (S \setminus \{u\}) = \hullc (S)$, a contradiction.
      
    \item To the contrary, if $S$ is convex, then $\hullc (S) = S$. Given that $\displaystyle\bigcup_{a \in S} \hullc (S \setminus \{a\}) = \hullc (S)$, we get that $S$ is $C$-dependent, a contradiction. \qedhere
\end{enumerate}
\end{proof}

\section{Computational Complexity}
\label{sec:complexity}

We show in this section that the problem related to the Carath\'{e}odory number in cycle convexity is NP-complete in the class of bipartite graphs. The decision problem is stated below.

\begin{problem}{\textsc{Carath\'{e}odory Number in Cycle Convexity}}\\
\textbf{Instance:} A graph $G$ and a positive integer $k$.\\
\textbf{Question:} Is $\car(G) \geq k$? 
\end{problem}


\begin{theorem}\label{theo:NPc_cth}
    \textsc{Carath\'{e}odory Number in Cycle Convexity} is NP-complete, even for bipartite graphs.
\end{theorem}

\begin{proof}
    The NP-membership is due to Araujo et al.~\cite{hull09}. The hardness part is done by a polynomial-time reduction from the NP-complete problem \textsc{One-In-Three 3-SAT (LO4)}~\cite{garey1979computers}. Such a problem receives as input a boolean formula $\phi = (X, \mathscr{C})$ in the conjunctive normal form (CNF), where $X$ is a set of variables and $\mathscr{C}$ is a collection of clauses such that, for every $c \in C$, $|c| = 3$. The question is whether there exists a truth assignment for $X$ such that each clause in $\mathscr{C}$ has exactly one true literal. The construction is inspired by~\cite{barbosa2012caratheodory}.
    
    From an instance $\phi = (X, \mathscr{C})$ of \textsc{One-In-Three 3-SAT} with $X = \{x_1, \dots, x_n\}$ and $\mathscr{C} = \{C_1, \dots, C_m\}$, we construct an instance $(G,k)$ of \textsc{Carath\'{e}odory Number in Cycle Convexity} following the next steps. We may assume that $m \geq 2$. 

     \begin{itemize}
    \item[(S1)] For every clause $C_i = (p \lor q \lor r)$, $1 \leq i \leq m$, we add to $G$ a clause gadget $G_{C_i}$ as defined in Figure~\ref{fig:gadgets_cth}(a). Let $A_C = \{c_i: 1\leq i \leq m\}$, $X_C = \{p_i^1, p_i^3, q_i^1, q_i^3, r_i^1, r_i^3 : 1 \leq i \leq m\}$, $Y_C = \{p_i^2, p_i^4, q_i^2, q_i^4, r_i^2, r_i^4 : 1 \leq i \leq m\}$, and $W_C = \{c_i^1, c_i^2 : 1 \leq i \leq m\}$.

    \item[(S2)] For every $1 \leq j < m$, we add to $G$ a subgraph $G_{D_j}$ denoted as in Figure~\ref{fig:gadgets_cth}(b). Let $A_D = \{ d_j^2, d_j^4, d_j^6, d_j^8 : 1 \leq j < m \}$, $B_D = \{d_j: 1\leq j < m\}$, $X_D = \{d_j^9, d_j^{11} : 1 \leq j < m \}$, $Y_D = \{d_j^1, d_j^3, d_j^5, d_j^7 : 1 \leq j < m \}$, and $W_D = \{d_j^{10}, d_j^{12}: 1 \leq j < m\}$. 
    
    Define $A = A_C \cup A_D$, $B = B_D$, $X = X_C \cup X_D$, $Y = Y_C \cup Y_D$, and $W = W_C \cup W_D$.
    
    
    \item[(S3)] For every pair of opposite literals $\ell \in C_i$ and $\overline{\ell} \in C_j$, for $1 \leq i < j \leq m$, we add nine vertices $\ell_0, \ell^5, \ell^6, \ell^7, \ell^8, \overline{\ell}^5, \overline{\ell}^6, \overline{\ell}^7, \overline{\ell}^8$ and twenty edges, obtained by all edges between the the following pair of vertex sets:
    \begin{multicols}{2}
    \begin{itemize}
        \item $(\{\ell^5\},\{\ell_i^2, \ell_i^4, \ell^6, \overline{\ell}^6\})$; 
        \item  $(\{\overline{\ell}^5\},\{\overline{\ell}_j^2, \overline{\ell}_j^4, \ell^6, \overline{\ell}^6\})$;
        \item $(\{\ell_i^2, \ell_i^4\}, \{\ell^7, \ell^8\})$; 
        \item  $(\{\overline{\ell}_j^2, \overline{\ell}_j^4\}, \{\overline{\ell}^7, \overline{\ell}^8\})$, and 
        \item $(\{\ell^0\},\{\ell^7, \ell^8, \overline{\ell}^7, \overline{\ell}^8\})$.
        \item[]
    \end{itemize}%
    \end{multicols}%
    Let $\ell^7, \ell^8, \overline{\ell}^7, \overline{\ell}^8 \in A$, $\ell^5, \overline{\ell}^5 \in X$, $\ell^0 \in Y$, and $\ell^6, \overline{\ell}^6 \in W$. See Figure~\ref{fig:construction_cth} for an example with literals $p$ and $\overline{p}$.
    

    \item[(S3)] Add a new vertex $c_0 \in B$ and the edge set $\{c_jc_{j+1} : 0 \leq j < m \} \cup \{c_0c_m\}$.

    \item[(S4)] Add the vertex set $Z = \{z,z'\}$ and the edges $zw, z'w$, for every $w \in W$.
    
    \item[(S5)] Define $k = 6m-3$.
    \end{itemize}
    
\begin{figure}[htb!]
    \centering
    \setlength{\fboxsep}{0pt} %
    \setlength{\fboxrule}{0pt}
    \fbox{
    \scalebox{0.99}{\input{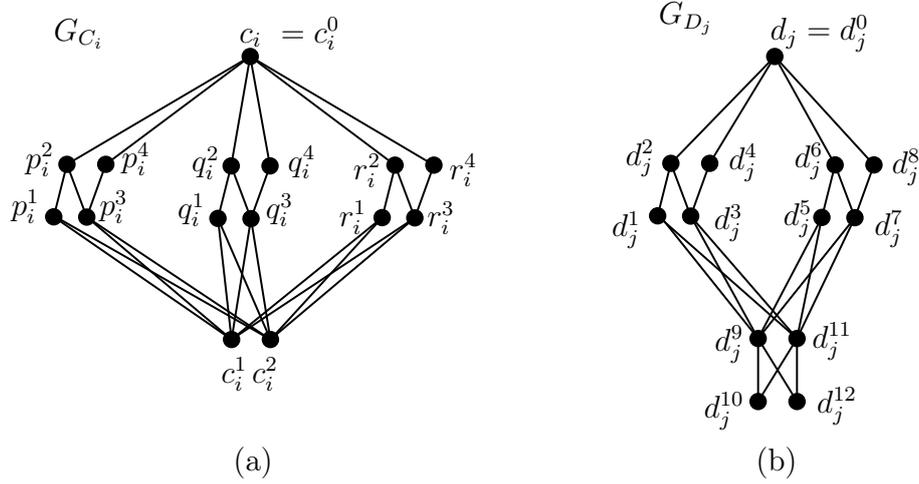}}
    }
    \caption{Subgraphs for the construction presented in Theorem~\ref{theo:NPc_cth}.
}
    \label{fig:gadgets_cth}
\end{figure}

   A sketch of the graph $G$ can be seen in Figure~\ref{fig:construction_cth}. By construction, $A \cup X \cup Z$ and $B \cup Y \cup W$ are independent sets. Then $V(G) = (A \cup X \cup Z) \cup (B \cup Y \cup W)$ implies that $G$ is bipartite. Let $L = \{\ell_0, \ell^s, \overline{\ell}^s: 1 \leq i \leq m, \, 5 \leq s \leq 8, \,\ell \in C_i, \,\ell \text{ has an opposite literal} \}$. From the Steps (S1)-(S5) we get $V(G) =\displaystyle \bigcup_{1 \leq i \leq m} V(G_{C_i}) \cup\displaystyle \bigcup_{1 \leq j < m} V(G_{D_j}) \cup L \cup \{c_0\} \cup Z$, which implies $|V(G)| = 15m+ 13(m-1) + 9 \cdot O(m^2) + 1 + 2  \in O(m^2)$, assuring a polynomial-time reduction.

\begin{figure}[htb!]
    \centering
    \setlength{\fboxsep}{0pt} %
    \setlength{\fboxrule}{0pt}
    \fbox{
    \scalebox{0.95}{\input{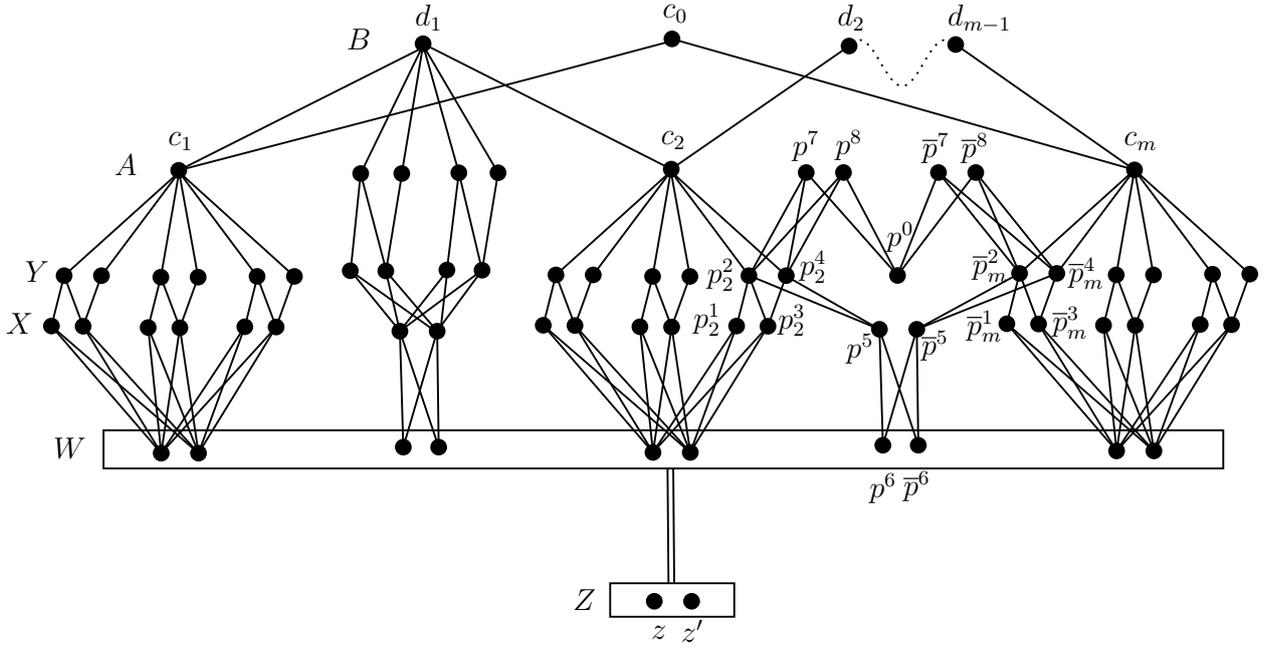}}
    }
    \caption{Sketch of the graph $G$ constructed for Theorem~\ref{theo:NPc_cth} from an instance of \textsc{One-In-Three 3-SAT} $\phi$ where the literal $p \in C_2$ and $\overline{p} \in C_m$. The double edges represents all the possible edges between the vertices of the two rectangles.
}
    \label{fig:construction_cth}
\end{figure}    

    Next, we show that $\phi$ is satisfiable with exactly one true literal per clause if and only if $G$ has Carathéodory independent set of order at least $k = 6m-3$.

    Let $t: X \to \{T,F\}$ be a truth assignment that satisfies all clauses in $\mathscr{C}$ with exactly one true literal per clause. Let $S = S_C \cup S_C$, where $S_C = \{\ell_i^2, \ell_i^3, \ell_i^4 : 1 \leq i \leq m, \ell \in C_i, \text{and } t(\ell) = T \}$ and $S_D = \{d_j^2, d_j^3, d_j^4: 1 \leq j < m \}$. Notice that $|S| = 3m + 3(m-1) = 6m-3$. We show that $S$ is a Carath\'{e}odory independent set of $G$.

    First, for every $1 \leq i \leq m$, $c_i$ lies in a cycle $\ell_i^2\ell_i^3\ell_i^4c_i$, for the true literal $\ell$ in $C_i$. Then $c_i \in \hullc (S)$ and $\{c_1, \dots, c_m\} \subseteq \hullc (S)$. Also, if $\ell \in C_i$ has an opposite literal, then each vertex $v \in \{\ell^5, \ell^7, \ell^8\}$ lies in the cycle $\ell_i^2\ell_i^3\ell_i^4v$, then $v \in \hullc (S)$. Consequently $\ell^0 \in \hullc (S)$ because of the cycle $\ell^7\ell^2_i\ell^8\ell^0$.     
    Further, for every $1 \leq j < m$, $d_j$ lies in a cycle $d_j^2d_j^3d_j^4d_j$, which implies that $d_j \in \hullc (S)$ and $\{d_1, \dots, d_{m-1} \} \subseteq \hullc (S)$. Further, given the cycle $c_0c_1d_1c_2d_2\dots d_{m-1}c_m$ we get that $c_0 \in \hullc (S)$ and $S' = \{c_0,c_1, \dots, c_m\} \cup \{d_1, d_2, \dots, d_{m-1} \} \cup \{ \ell^0, \ell^5, \ell^7, \ell^8 : \ell \in C_i$ has an opposite literal in $\phi \} \cup S \subseteq \hullc (S)$. By construction, no vertex in $V(G) \setminus S'$ has two neighbors in $S'$, which implies that $S' = \hullc (S)$. 
    
    Next, we evaluate $\hullc (S \setminus \{a\})$, for every $a \in S = S_C \cup S_D$. First, let $a = \ell_i^s \in S_C$, for some $i \in \{1, \dots, m\}$, $\ell \in C_i$ such that $\ell$ is true, and $s \in \{2,3,4\}$. In this case, for every $i' \in \{1, \dots, m\}\setminus \{i\}$, $c_{i'} \in \hullc (S \setminus \{a\})$ as well as $\{\ell^0, \ell^5, \ell^7, \ell^8 : \ell \in C_{i'} \text{ has an opposite literal}\} \subseteq \hullc (S \setminus \{a\})$, Next, $d_j \in \hullc (S \setminus \{a\})$, for every $1 \leq j < m$, and the main point is that $c_i \notin \hullc (S \setminus \{a\})$. Recall that $N_G(c_0) = \{c_1, c_m\}$. If $a = \ell_1^s$ (resp. $a = \ell_m^s$) it is clear that $c_1 \notin \hullc (S)$ (resp. $c_m \notin \hullc (S)$), which implies that $c_0 \notin \hullc (S)$. Otherwise, we have that $N_G(c_0) = \{c_1, c_m\} \subseteq \hullc (S)$, but in the graph induced by $\hullc (S)$ the vertices $c_1$ and $c_m$ belong to distinct connected components, which implies that $c_0 \notin \hullc (S \setminus \{a\})$. The case $a = d_j^s \in S_2$, for some $j \in \{1, \dots, m-1\}$ and $s \in \{2,3,4\}$ is analogous. Then $c_0 \notin \displaystyle\bigcup_{a \in S} \hullc (S \setminus \{a\})$. As we have seen in the previous paragraph, $c_0 \in S' = \hullc (S)$, and then $S$ is a Carath\'{e}odory independent set.

    For the converse, let $S$ be a Carath\'{e}odory independent set of $G$ with $|S| \geq 6m-3$. We begin with a property of the construction regarding the vertices in $Z$.

    \medskip
    \noindent \textbf{Claim~1.} Let $t \geq 0$ be an integer. If $z,z' \in I_{cc}^t(S)$, then $\hullc (S) = V(G)$.

    \smallskip
    \noindent Suppose that $z,z' \in I_{cc}^t(S)$, for some integer $t \geq 0$. By hypothesis, $S$ is a Carathéodory independent set, then Lemma~\ref{lemma:cth_ind_set_cc}(c) states that $G[\hullc (S)]$ is connected. Thus, there exists $t' \geq 0$ such that some vertex $w \in N_G(\{z,z'\}) \subseteq W$ belongs to $I_{cc}^{t'}(S)$. By construction, the false twins $z$ and $z'$ are both adjacent to every vertex in $W$. Thus, the path $zwz'$ forms a $4$-cycle with each vertex $w' \in W \setminus \{w\}$. This gives $W \subseteq \hullc (S)$. Further, we have that every vertex $x \in X$ has two neighbors $w_1, w_2 \in W$. Since $w_1, w_2, z \in \hullc (S)$, the set $\{w_1, z, w_2, x\}$ induces a $C_4$ in $G$, then $x \in \hullc (S)$; in consequence $X \subseteq \hullc (S)$. In addition, for $1 \leq j < m$, the path $d_j^{9}d_j^{10}d_j^{11}$ forms a $4$-cycle with each vertex in $\{d_j^1, d_j^{3}, d_j^{5}, d_j^{7}\}$, then $d_j^1, d_j^{3}, d_j^{5}, d_j^{7} \in \langle S \rangle$.

    Next, let $U_1 = \{\ell_i^2: \ell \in C_i \} \cup \{d_j^2 : 1 \leq j < m \} $ and $U_2 = \{\ell_i^4: \ell \in C_i$  and $\ell$ has an opposite literal$\}$.        
    For $u \in U_1 \cup U_2$, the construction implies that $u$ has two neighbors $u_1, u_2 \in X \cup \{d_j^1, d_j^{3}, d_j^{5}, d_j^{7} : 1 \leq j < m \}$. Particularly if $u \in U_1$, $u$ lies in a cycle $u_1wu_2u$, for some $w \in N_G(\{u_1,u_2\}) \cap (W \cup \{d_j^9, d_j^11\})$.  The case $u \in U_2$, implies that $u$ belongs to a cycle $u_1wzw'u_2u$ for $w, w' \in N_G(\{u_1,u_2\}) \cap W$. Hence $u \in \hullc (S)$ and $U_1 \cup U_2 \subseteq \hullc (S)$.

    Now, consider any two literals $p,q \in C_i$. Since $p_i^2, q_i^2 \in \hullc (S)$, we have that $c_i \in \hullc (S)$ because of the cycle $p_i^2p_i^3c_i^1q_i^3q_i^2c_i$, as $p_i^3,q_i^3, c_i^1 \in \hullc (S)$. Given the cycle $c_ip_i^2p_i^3p_i^4$, we get that $p_i^4 \in \hullc (S)$. This shows that $V(G_{C_i}) \subseteq \hullc (S)$ for every $i \in \{1, \dots, m\}$. Next, consider $V(G_{D_j})$ for $j \in \{1, \dots, m-1\}$. Since $\{d_j^1, d_j^3, d_j^5, d_j^7, d_j^9, d_j^{11}\} \subseteq \hullc (S)$, the previous reasoning applies here as well. By replacing the vertices $p$ and $q$ with $d_j^2$ and $d_j^6$, respectively, and using analogous arguments, we conclude that $V(G_{D_j}) \subseteq \hullc (S)$. Since $\{c_1, \dots, c_m, d_1, \dots, d_{m-1}\} \subseteq \hullc (S)$, the cycle $c_1d_1c_2d_2\dots d_{m-1}c_mc_0$ implies that $c_0 \in \hullc (S)$.    
    Additionally, for $u \in U_3 = \{\ell^7, \ell^8, \overline{\ell}^7, \overline{\ell}^8 : 1 \leq i \leq m, \,\ell \in C_i \text{ such that } \ell \text{ has an opposite literal} \}$, the cycle $\ell_i^2\ell_i^3\ell_i^4u$ assures $u \in \hullc (S) $, then we get $U_3 \in \langle S \rangle$. Also, $\ell^0$ lies in a cycle $\ell^7\ell^2\ell^8\ell^0$, then $\ell^0 \in \hullc (S)$. This gives $\hullc (S) = V(G)$ as stated.

    \medskip
    
    In the next claims we show some restrictions on which parts of $G$ can contain vertices of the Carathéodory independent set $S$. Recall that $|S| \geq 6m-3$. Define $F = Y_D \cup X \cup W \cup Z$.

    \medskip
    \noindent \textbf{Claim~2.} It holds that $G[F \cap S]$ is a $P_3$-free graph. Moreover every connected component of $G[F \cap S]$ is isomorphic to either a $K_1$ or a $K_2$.

    \smallskip
    \noindent Suppose by contradiction that $S' = \{f_1,f_2,f_3\}$ induces a $P_3$ in $G[F \cap S]$. By construction, every vertex in $F$ has a false twin in $G[F]$, then, regarding to $f_2$ there is $f_4 \in F \cap N_G(\{f_1,f_3\})$. Thus, the cycle $f_1,f_2,f_3,f_4$ implies that $f_4 \in \hullc (S')$. Again, by construction, either immediately $Z \subseteq \hullc (S')$ or $z,z'$ has two neighbors in $\{f_1, f_2, f_3, f_4 \}$, which also gives further $Z \subseteq \hullc (S')$. By Claim~1, $\hullc (S') = V(G)$ and since $m \geq 2$,  $|S| \geq 6m-3 \geq 9$, which implies that $S'$ is a proper subset of $S$, contradicting Lemma~\ref{lemma:cth_ind_set_cc}(d). Hence $G[F \cap S]$ is a $P_3$-free graph, i.e, a disjoint union of cliques. Since $G$ is bipartite, the maximum clique of $G$ has order two, then every connected component of $G[F \cap S]$ is either a $K_1$ or a $K_2$. 

    \medskip
    \noindent \textbf{Claim~3.} Let $V_1 = \{\ell_i^3 : \ell \in C_i,  1 \leq i \leq m\}$, $V_2 = \{d_j^3, d_j^7: 1 \leq j < m \}$, and $V_3 = \{ \ell^5, \overline{\ell}^5: \ell$ has an opposite literal$\}$. If $v \in F \cap S$, then $v \in V_1 \cup V_2 \cup V_3$.
    
    \smallskip
    \noindent For the opposite, suppose that $v \in (F \cap S) \setminus  (V_1 \cup V_2 \cup V_3)$. Remind that Lemma~\ref{lemma:cth_ind_set_cc}(f) implies that $\hullc (S)$ induces a graph with no leaves and Claim~2 implies that $F \cap S$ induces a $P_3$-free graph. If $v \in Z \cup W \cup \{d_j^9, d_j^{11}: 1 \leq j < m\}$, given that $G[\hullc (S)]$ is connected (by Lemma~\ref{lemma:cth_ind_set_cc}(c)), $v$ has two neighbors $v_1,v_2$ in $F \cap \langle S \rangle$, which induce a $P_3: v_1vv_2$, a contradiction. Then we obtain that $\hullc (S) \cap (Z \cup W \cup \{d_j^9, d_j^{11}: 1 \leq j < m\}) = \emptyset$. 
    If $v \in \{\ell_i^1: \ell \in C_i, 1 \leq i \leq m\} \cup \{d_j^1, d_j^5 : 1 \leq j < m\}$, then by construction $v$ has at most one neighbor in $(V(G) \setminus F) \cap \hullc (S)$, thus $v$ is a leaf or an isolated vertex in $\hullc (S )$, a contradiction.

    \medskip
    For the rest of the proof, let $Q_{C_i} =  \{c_i\} \cup \{ \ell_i^2, \ell_i^3, \ell_i^4: \ell \in C_i \} \cup \{\ell^0, \ell^5, \ell^7, \ell^8 :  \ell \in C_i$  which has an opposite literal $\overline{\ell}$ in $\phi\}$ and $Q_{D_j} = \{d_j, d_j^2, d_j^3, d_j^4, d_j^6, d_j^7, d_j^8\}$. In the next, Claim~4.1 expresses some properties of a subset of $Q_{C_i}$ associated to a literal $p \in C_i$, as well as Claim~4.2 captures a similar idea for $Q_{D_j}$.

    \medskip
    \noindent \textbf{Claim~4.1.} Let $v \in S$ and for a literal $p \in C_i$, we define $Q_{C_i}(p) = \{c_i\} \cup \{ p_i^2, p_i^3, p_i^4\} \cup \{p^0, p^5, p^7, p^8 :  p$  has an opposite literal $\overline{p}$ in $\phi\}$.  If $v \in Q_{C_i}(p)$, then the following holds:
    \begin{itemize}
        \item[(a)] $S \cap (Q_{C_i} \setminus Q_{C_i}(p)) = \emptyset$, 
        \item[(b)] $|S \cap Q_{C_i}(p)| = 3$, and
        \item[(c)] for every clause $C_{i'}$ that contains $\overline{p}$,  $S \cap (Q_{C_{i'}}(\overline{p}) \setminus \{p^0 \}) = \emptyset$.
    \end{itemize}

    \smallskip
    \noindent Let $v \in S \cap Q_{C_i}(p)$ according to the assumption. For Item~(a), suppose by contradiction that there exists $v' \in S \cap (Q_{C_i} \setminus Q_{C_i}(p))$. Recall that $S$ is a $C$-independent set, so let $u \in \hullc (S) \setminus\displaystyle \bigcup_{a\in S}\hullc (S\setminus\{a\})$. Since $G[\hullc (S)]$ is connected (Lemma~\ref{lemma:cth_ind_set_cc}(c)), $v$ is not a leaf in $G[\hullc (S)]$ (Lemma~\ref{lemma:cth_ind_set_cc}(f)), and $\hullc (S) \cap W = \emptyset$ (Claim~3), we have that $v$ has two neighbors $v_1, v_2 \in \hullc (S) \cap Q_{C_i}$. The construction implies that $v_1,v_2 \in S$, then $|S \cap Q_{C_i}(p)| =  |\{v,v_1,v_2 \}| = 3$, which gives Item (b). In consequence $\hullc( \{v,v_1,v_2 \} ) = Q_{C_i}(p)$. Since $c_i \in \hullc (S \cap Q_{C_i})$ and $c_i \in  \hullc (S \cap Q_{C_i}(p))$ we get that $u \in \hullc (S )$ as well as $u \in \hullc (S \setminus \{v'\} )$, a contradiction. 

    For Item (c), given that $\hullc (\{v,v_1,v_2 \} ) = Q_{C_i}(p)$, if $v'' \in S \cap (Q_{C_{i'}}(\overline{p}) \setminus \{p^0 \})$, then $(Q_{C_{i'}}(\overline{p}) \setminus \{p^0 \}) \subseteq \hullc( \{v,v_1,v_2,v'' \} )$. Thus, the path $p^5p_i^4p^7p^0\overline{p}^7\overline{p}_{i'}^4\overline{p}^5$ implies that $p^6\overline{p}^6 \in \hullc( \{v,v_1,v_2,v'' \} )$, a contradiction since $S \cap W = \emptyset$ (recall Claim~3).

    \medskip
    \noindent \textbf{Claim~4.2.} Let $v \in S$. Define $Q_{D_j}(0) = \{d_j, d_j^2, d_j^3, d_j^4 \}$ and $Q_{D_j}(1) = \{d_j, d_j^5, d_j^6, d_j^7 \}$. For $s \in \{0,1\}$, if $v \in Q_{D_j}(s)$, then $S \cap (Q_{D_j} \setminus Q_{D_j}(s)) = \emptyset$ and $|S \cap Q_{D_j}(s)| = 3$.
    
    \smallskip
    \noindent The proof is analogous to that of Claims~4.1(a) and (b).

    \medskip
    Finally, Claim~5 is the last needed tool to define a truth assignment for $\phi$.
    
    \medskip
    \noindent \textbf{Claim~5.}  For every $1 \leq i \leq m$ and $1 \leq j < m$, both $S \cap Q_{C_i}$ and $S \cap Q_{D_j}$ are non-empty.

    \smallskip
    \noindent Let $T = (Q_{C_1} \cup \ldots \cup Q_{C_m}) \cup (Q_{D_1} \cup \ldots \cup Q_{D_{m-1}})$ and recall that $|S| = 6m-3$. Suppose by contradiction that $S \cap Q_{C_i} = \emptyset$, for some $i \in \{1, \dots, m\}$. By Claim~3, $S \cap (Z \cup W) = \emptyset$, and by Claims~4.1 and~4.2, $|S \cap Q_{C_i}| \leq 3$ and $|S \cap Q_{D_j}| \leq 3$. Thus, we obtain that $|S| \leq |(T \cap S) \cup \{c_0\}| \leq 3(m-1)+3(m-1)+1 = 6m-5$, a contradiction. Therefore $S \cap Q_{C_i} \neq \emptyset$. The reasoning for $S \cap Q_{D_j}$ is similar.

    \medskip

    By Claims~4.1, 4.2, and 5, we obtain that for every $i \in \{1, \dots, m\}$ and $j \in \{1, \dots, m-1\}$,  $|S \cap Q_{C_i}|, |S \cap Q_{D_j}| \in \{1,2,3\}$. Then it is clear that $\{ c_1, \dots, c_m\} \cup \{d_1, d_2, \dots, d_{m-1} \} \subseteq \hullc (S)$ and subsequently $c_0 \in \hullc (S )$. By contradiction, suppose that there is some $i \in \{1, \dots, m\}$ such that $|S \cap Q_{C_i}| = 1$. This implies that $|S| \leq 3(m-1)+3(m-1)+1+1 = 6m-4$, a contradiction. So, we have that $|S \cap V(G_{C_i})| \geq 2$, $|S \cap V(G_{D_j})| \geq 2$. In particular, if $c_i \in S$ and either $p^0$ or $p^5$ belong to $S$, then Lemma~\ref{lemma:cth_ind_set_cc}(f) implies that $S \cap (\{p_i^2, p_i^3, p_i^4, p^7, p^8 \}) \neq \emptyset$. Then defining a truth assignment where $t(p) = T$, for the literal $p \in C_i$ such that $S \cap (\{p_i^2, p_i^3, p_i^4, p^5, p^7, p^8 \}) \neq \emptyset$, and $t(p') = F$, for $p' \in C_i$, $p \neq p'$, we obtain a solution to $\phi$ with exactly one true literal per clause.    
\end{proof}

\section{Extremal Carath\'{e}odory Numbers}
\label{sec:extremal}

In this section, we investigate the Carathéodory number for small and large values, namely up to $3$, and near to the number of vertices of the graph.

\begin{proposition}\label{cthcycle}
    Let $G$ be a graph. It holds that: 
    \begin{enumerate}[label=\textnormal{(}\alph*\textnormal{)}]
        \item $\car(G)=1$ if and only if $G$ is a forest.    
        \item If $\{u,v\}$ is a hull set of $G$ for every edge $uv \in E(G)$, then $\car(G)=2$.        
    \end{enumerate}
\end{proposition}

\begin{proof}
    \begin{enumerate}[label=(\alph*)]
    \item If $G$ is acyclic, then every subset of $V(G)$ is convex. Consequently, the only $C$-independent sets in $G$ are the singleton vertex sets. 
Now, suppose that $G$ contains a cycle. Let $C = v_1v_2\cdots v_\ell$, with $\ell \geq 3$, be an induced cycle in $G$. We may choose $S = \{v_2, \dots, v_\ell\}$. Then it is clearly that $v_1 \in \hullc (S)$. On the other hand, for each $i \in \{2, \dots, \ell\}$, we have
$
v_1 \notin \hullc (S \setminus \{v_i\} ).
$
Thus,
$
v_1 \notin\displaystyle \bigcup_{i \in \{2, \dots, \ell\}} \hullc (S \setminus \{v_i\}).$ It follows that $S$ is a $C$-independent set in $G$, and hence $\car(G) \geq |S| = \ell - 1 \geq 2$.
    
    \item Let $G$ be a graph in which every pair of adjacent vertices forms a hull set. Suppose, for a contradiction, that $G$ contains a $C$-independent set $S$ with $|S| \geq 3$. By Lemma~\ref{lemma:cth_ind_set_cc}(a), the induced subgraph $G[S]$ contains at least one edge, say $uv$. By assumption, every pair of adjacent vertices is a hull set, and hence $\hullc( \{u,v\} ) = V(G)$. Since $\{u,v\} \subsetneq S$, this contradicts Lemma~\ref{lemma:cth_ind_set_cc}(d). Hence $\car(G) \leq 2$. Moreover, since no pair of adjacent vertices forms a hull set in a forest, it follows that $G$ is not a forest. By Item~(a), we have $\car(G) \geq 2$. Combining this with the previous bound, we find that $\car(G) = 2$. \qedhere
    \end{enumerate}
\end{proof}
The converse of Proposition~\ref{cthcycle}(b) does not hold in general. As an illustration, consider the house graph, obtained from a cycle on four vertices by adding a fifth vertex adjacent to two consecutive vertices of the cycle. It is straightforward to verify that this graph has the Carathéodory number $2$, yet there exist edges that are not hull sets in $G$. 

Next, we deal with the disjoint union and join operations on general graphs.

\begin{proposition}\label{prop:max_cth}
Let $G_1$ and $G_2$ be graphs. Then $\car(G_1 \cup G_2) = \max\{\car(G_1), \car(G_2)\}.$
\end{proposition}

\begin{proof} 
Let $S$ be a $C$-independent set of $G = G_1 \cup G_2$. By Lemma~\ref{lemma:cth_ind_set_cc}(c) $G[\hullc (S)]$ is connected. Then either $S \subseteq V(G_1)$ or $S \subseteq V(G_2)$ and the lower bound is immediate.
\end{proof}

For the next, we denote by $\trivial(G)$ and $\nt(G)$ the number of trivial and nontrivial connected components of a graph $G$, respectively.

\begin{proposition}\label{prop:join}
Let $G_1$ and $G_2$ be graphs. Then $\car(G_1 + G_2) \leq 3$. In particular,
$$
\car(G_1 + G_2) =
\begin{cases}
\begin{aligned}
1, & \  \text{if } |V(G_1)| = 1 \text{ and } |E(G_2)| = 0; && \quad \textnormal{(i)} \\[0.15cm]
3, & \ \text{if } |V(G_1)| = 1 \text{ and } \nt(G_2) \geq 2, \text{ or } && \quad \textnormal{(ii)} \\
   & \ \quad |V(G_1)| \geq 2,\; \trivial(G_i) \geq 1 \text{ and } \trivial(G_j) \geq 2, \text{ for } i,j \in \{1,2\}, i \neq j; && \quad \textnormal{(iii)} \\
2, &\ \text{otherwise}. && \quad \textnormal{(iv)}
\end{aligned}
\end{cases}
$$
\end{proposition}

\begin{proof}
    For the upper bound, let $S$ be a $C$-independent set of $G = G_1 + G_2$ and assume by contradiction that $|S| \geq 4$, say $u_1, \dots, u_4 \in S$. First, suppose that $S$ intersects at most one of $G_1$ and $G_2$, say $S\subseteq V(G_1)$. Since $S$ is a non-convex set of $G$, the induced subgraph of $S$ in $G$ has at least one edge, say $u_1u_2$. This implies that $V(G_2) \subseteq \intv(\{u_1,u_2\})$ and consequently $V(G_1) \setminus \{u_1,u_2\} \subseteq \intv(V(G_2) \cup \{u_1,u_2\})$. 
    Then $\{u_1, u_2\}$ is a hull set of $G$, a contradiction to the fact that $S$ is $C$-independent in $G$. This proves that $S$ intersect both $G_1$ and $G_2$. Since the induced subgraph of $S$ in $G$ has no cycles, without loss of generality, we may assume that $u_1,u_2,u_3 \in V(G_1)$ and $u_4 \in V(G_2)$. 
    If $|V(G_2)| \geq 2$, then $V(G_2) \setminus \{u_4\} \subseteq \intv(\{u_1,u_2,u_4\})$. Then $V(G_1) \setminus \{u_1,u_2\} \subseteq \intv(V(G_2) \cup \{u_1,u_2\})$. Thus,  $\{u_1, u_2, u_4\}$ is a hull set of $G$, a contradiction.
    Now, consider that $|V(G_2)| = |\{u_4\}| = 1$. By Lemma~\ref{lemma:cth_ind_set_cc}(f), we have that $\hullc (S)$ induces a graph with no leaves. Then, every vertex $u \in S \cap V(G_1)$ lies in a nontrivial connected component of $G_1$, say $G_u$. This implies that $\hullc (S) =\big( \bigcup_{u \in S \cap V(G_1)} V(G_u) \big) \cup \{u_4\} =  \bigcup_{a \in S} \hullc (S \setminus \{a\} )$, then $S$ is $C$-dependent, a contradiction. Thus, the upper bound $\car(G) \leq 3$ holds.

    For the equalities, first, it is enough to recall that Condition~(i) imply that $G$ is a star. Therefore, by Proposition~\ref{cthcycle}(a) $\car(G) = 1$.

    For Condition~(ii) consider $V(G_1) = \{u\}$ and $v,w$ two vertices that lie in distinct nontrivial components of $G_2$, say $G_v$ and $G_w$, respectively. Let $v' \in N_{G_2}(v)$ and $w' \in  N_{G_2}(w)$. Define $S = \{v,v',w\}$. Since $u$ is adjacent to every vertex in $V(G_2)$, we have that $u \in \intv(\{v,v'\})$. In consequence $w' \in \hullc (S ) = V(G_u) \cup V(G_v) \cup \{u\}$. On the other hand,  $w' \notin\displaystyle \bigcup_{a \in S} \hullc (S \setminus \{a\} ) = V(G_u) \cup \{w\}$. Then $S$ is a $C$-independent set and $\car(G) \geq 3$. 
    
    For Condition~(iii) let $|V(G_1)| \geq 2$, $u$ be an isolated vertex in $G_1$, $u' \in V(G_1) \setminus \{u\}$, and $v,w$ be two isolated vertices in $G_2$ (by symmetry we fixed $i = 1$ and $j = 2$). Let $S = \{u,v,w\}$. The cycle $v,u,w,u'$ implies that $u' \in \hullc (S)$. On the other hand, since $N_G(u) \subseteq V(G_2)$ and $N_G(v), N_G(w) \subseteq V(G_1)$, the edges $uv, uw$ do not have a common neighbor, then $\{u,v\}$ and $\{u,w\}$ are convex sets. Moreover, $\intv(\{v,w\}) = \{v,w\}$. Thus $u' \notin \displaystyle\bigcup_{a \in S} \hullc (S \setminus \{a\} )$. This shows that  $S$ is a $C$-independent set and $\car(G) \geq 3$. 

    Combining Conditions~(ii) and~(iii) with the upper bound established in the first paragraph, we conclude that $\car(G) = 3$. Next, let us analyze the contemplated and missing cases.
    
    Notice that if $|V(G_1)| = 1$, the case $|E(G_2)| = 0$ is captured by Condition~(i) and the case $|E(G_2)| \geq 1$ implies that $G_2$ has at least a nontrivial component, which is captured by Condition~(ii). Otherwise, $|V(G_1)| \geq 2$. If $G_1$ has a trivial component and $G_2$ has two trivial components (or vice versa), we get Condition~(iii). So, remain to consider the case $\trivial(G_i) < 1$ or $\trivial(G_j) < 2$, for distinct $i,j \in \{1,2\}$, that is, $\trivial(G_i) = 0$ or $\trivial(G_j) \in \{0,1\}$.

    If $\trivial(G_i) = 0$, say $i = 1$, then $G_1$ has no isolated vertices. If $|V(G_2)| = 1$, then by symmetry we fall back to Conditions~(i) and~(ii). Hence, we may assume that $|V(G_2)| \geq 2$. In the join $G = G_1 + G_2$, every edge $uv \in E(G)$ belongs to a $K_3$. Since $G$ is connected and $|V(G_1)|, |V(G_2)| \geq 2$, it follows that $\{u,v\}$ is a hull set of $G$, for every $uv \in E(G)$. Therefore, Proposition~\ref{cthcycle} implies that $\car(G) = 2$.
    
    Now, consider the remaining case $\trivial(G_j) \in \{0,1\}$, say $j = 2$. The case $\trivial(G_2) = 0$ was already handled above, so assume that $G_2$ has exactly one isolated vertex. Since $|V(G_1)|, |V(G_2)| \geq 2$, and the case $\trivial(G_1) = 0$ was also treated, we may assume that $\trivial(G_1) = 1$. Let $S$ be a $C$-independent set of $G$, and let $\{u\}$ and $\{v\}$ be the trivial components of $G_1$ and $G_2$, respectively. By construction, the edge $uv$ does not belong to any $K_3$, and hence $\{u,v\}$ is convex. By Lemma~\ref{lemma:cth_ind_set_cc}(g), we have $S \neq \{u,v\}$. On the other hand, the join implies that every edge $xy \in E(G) \setminus \{uv\}$ belongs to a $K_3$. Thus, any set $S$ with $|S| = 2$ and $S \neq \{u,v\}$ is $C$-independent. Moreover, by Lemma~\ref{lemma:cth_ind_set_cc}(d), no proper subset of $S$ is a hull set, which implies that $|S| \leq 2$. Therefore, $\car(G) = 2$.
\end{proof}

In the following, we determine the Carathéodory number of complete bipartite graphs, showing that the Carathéodory numbers corresponding to $\Delta$-convexity (which is $1$, see~\cite{anand2025caratheodory}) and cycle convexity (which is $3$, see Corollary~\ref{com_bi}) are intrinsically different. As a further remark, complete multipartite graphs fall into Case~(b) of Proposition~\ref{cthcycle}, and thus have cycle Carathéodory number equal to $2$.

\begin{corollary}\label{com_bi} For integers $m,n \geq 2$, $\car(K_{m,n})=3$.
\end{corollary}

\begin{proof}
    Let $K_{m,n} \simeq \overline{K}_m + \overline{K}_n$ be a complete bipartite graph. Since $m \geq 2$, $\trivial(\overline{K}_m) \geq 1$ and $\trivial(\overline{K}_n) \geq 2$, by Proposition~\ref{prop:join}(iii) we get that $\car(K_{m,n})=3$.
\end{proof}

A graph $G$ is $P_4$\textit{-sparse} if no set of five vertices in $G$ induces more than one induced $P_4$. The $P_4$-sparse graphs and can be described by disjoint unions, joins and spiders and graphs generalize both cographs and $P_4$-reducible graphs~\cite{hoang1985perfect,jamison1992recognizing}. To determine an upper bound to the Carathéodory number of $P_4$-sparse graphs we use the definition and the result in sequel.

A \textit{spider} is a graph whose vertex set can be partitioned into $I$, $K$, and $R$, where 
$I = \{s_1, \ldots, s_k\}$ ($k \geq 2$) is an independent set, 
$K = \{k_1, \ldots, k_k\}$ is a clique, and either $s_i$ is adjacent to $k_j$ if and only if $i = j$ (a \textit{thin spider}), or $s_i$ is adjacent to $k_j$ if and only if $i \neq j$ (a \textit{thick spider}). 
The set $R$, called the \textit{head}, is allowed to be empty and, if it is not, then every vertex in $R$ is adjacent to all vertices in $K$ and non-adjacent to all vertices in $I$. 
The triple $(I,K,R)$ is called the \textit{spider partition}, and can be found in linear time~\cite{jamison1992recognizing}.

\begin{lemma}\label{lemma:spider}
Let $G = (I,K,R)$ be a spider with which is not a cograph. Then $\car(G) \leq 2$.
\end{lemma}

\begin{proof}
If $G$ is isomorphic to a tree, then $\car(G) = 1$ by Proposition~\ref{cthcycle}(a).
Thus, assume that $G$ has a cycle. Since $G$ is not a cograph, we have that $|I| = |K| \geq 2$. 
Let $S$ be a $C$-independent set of $G$.
By definition, $G[K \cup R]$ is a join $K + R$. Then every edge $uv \in E(G[K \cup R])$ belongs to a $K_3$ (even if $R = \emptyset)$ and consequently $\{u,v\}$ is a hull set of $G[K \cup R]$. Then Proposition~\ref{cthcycle}(b) implies that $\car(G[K \cup R]) = 2$. By Lemma~\ref{lemma:cth_ind_set_cc}(f), the graph $G[\langle S \rangle]$ has no leaves.
If $G$ is a thin spider, then every vertex in $I$ has degree one in $G$. Hence, no vertex of $I$ can belong to $\hullc (S)$, and therefore $S \cap I = \emptyset$. It follows that $S \subseteq K \cup R$, and by Proposition~\ref{cthcycle}(b), we have $\car(G) = 2$.
Now, assume that $G = (I,K,R)$ is a thick spider, with $|S| = |K| \geq 3$. In this case, every edge $xy \in E(G)$ with $x \in I$ and $y \in K$ belongs to a $K_3$. Hence, $\{x,y\}$ is a hull set of $G$. Therefore, again by Proposition~\ref{cthcycle}(b), we conclude that $\car(G) = 2$.
\end{proof}

\begin{corollary}
If $G$ is a $P_4$-sparse graph, then $\car(G) \leq 3$.
\end{corollary}

\begin{proof}
The proof is by induction on $|V(G)|$. If $|V(G)| \leq 1$, the result is trivial. Since $G$ is $P_4$-sparse, we know by~\cite{hoang1985perfect} that $G$ is either:
\begin{enumerate}[label={\normalfont(\alph*)}]
    \item the disjoint union of two $P_4$-sparse graphs;
    \item the join of two $P_4$-sparse graphs;
    \item a spider whose head is either empty or a $P_4$-sparse graph.
\end{enumerate}

In Case~(a), the conclusion follows from Proposition~\ref{prop:max_cth} and the induction hypothesis. In Case~(b), the conclusion follows directly from Proposition~\ref{prop:join}. In Case~(c), it follows from Lemma~\ref{lemma:spider} and the induction hypothesis. Therefore, $\car(G) \leq 3$.
\end{proof}

The split graphs also have a small Carathéodory number.  A \textit{split graph} $G$ is a graph whose vertex set $V(G)$ can be partitioned into a clique $K$ and an independent set $I$.

\begin{proposition}\label{split} Let $G = (K \cup I, E)$ be a split graph. Then $\car(G) \leq 2$.
\end{proposition}

\begin{proof}
    If $G$ is a tree, the result $\car(G) = 1$ is known from Proposition~\ref{cthcycle}(a). So assume the opposite. Since $G$ has a cycle, which is, by definition a $K_3$, we may assume that $|K| \geq 2$. Let $L = \{v \in I :
     \deg(v) = 1\}$. It is easy to see that every edge $uv \in E(G)$ such that $u,v \notin L$ is a hull set of $G - L$. Since $S \cap L = \emptyset$ (by Lemma~\ref{lemma:cth_ind_set_cc}(f)), the conclusion $\car(G) = 2$ holds.
\end{proof}

It is clear that there is no graph $G$ of order $n \geq 2$ with Carath\'{e}odory number $n$. We therefore turn to the characterization of graphs of order $n \geq 3$ whose Carath\'{e}odory number is equal to $n-1$.
   \begin{theorem}\label{caratheodorycycle}
   Let $G$ be a connected graph of order $n\geq 3$. Then $\car(G) = n-1$ if and only if $G =C_n$.
\end{theorem}

\begin{proof} First, suppose that there exists a vertex $u \in V(G)$ such that the set $S = V(G) \setminus \{u\}$ is $C$-independent in $G$. Then, by Proposition~\ref{cthcycle}(a),  $G$ contains at least one cycle, say $C$. Then it follows from Lemma~\ref{lemma:cth_ind_set_cc}(e) that $u \in V(C)$. Suppose that $V(G) \neq V(C)$. Then, for any vertex $v \in V(G) \setminus V(C)$, we have 
$
V(C) \subseteq \hullc (S \setminus \{v\} ).
$
In addition, for any vertex $w \in V(C)$, we have that 
$
V(G) \setminus V(C) \subseteq \hullc (S \setminus \{w\} ).$ Hence, $S$ is $C$-dependent, a contradiction. Therefore, $G$ must be a cycle. Conversely, it is clear that $\car(C_n) = n-1$.
\end{proof}




Next, we characterize all connected graphs $G$ of order $n \geq 4$ with Carath\'{e}odory number $n-2$. For this purpose, we introduce the following four families of graphs.

\begin{itemize}

\item The family $\mathscr{F}_1$ consists of all graphs obtained from a cycle by adding a new vertex $v$ and joining $v$ to at least one vertex of the cycle.

\item The family $\mathscr{F}_2$ consists of all $2$-connected bicyclic graphs.


\item The family $\mathscr{F}_3$ consists of all graphs obtained from a $(u,v)$-path and an edge $xy$ by adding all possible edges between the sets $\{x,y\}$ and $\{u,v\}$.


\item The family $\mathscr{F}_4$ consists of all connected bicyclic graphs that are not $2$-connected.

\end{itemize}

As a remark, the graphs in $\displaystyle\bigcup_{i=1}^4 \mathscr{F}_i$ can be recognized in polynomial time. This follows from the fact that the definitions of these families rely on detecting cycles, testing $2$-connectedness, and verifying join operations, all of which can be performed in polynomial time.

\begin{theorem}
  Let $G$ be a connected graph of order $n \geq 4$. Then $\car(G) = n-2$ if and only if 
$G \in \displaystyle\bigcup_{i=1}^4 \mathscr{F}_i.$
    \end{theorem}

\begin{proof}
Let $G$ be a graph with $\car(G)=n-2$, and let $S$ be a maximum $C$-independent set such that $x,y \notin S$. Let $T_1, T_2, \ldots, T_k$ ($k \geq 1$) denote the components of the induced subgraph $G[S]$. Recall that each $T_i$ is a tree.

\medskip
\noindent {\bf Case 1.} $xy\in E(G)$.

Since $S$ is a non-convex set in $G$, either $x$ or $y$, say $x$, has at least two neighbours in some $T_k$, say $T_1$. Suppose that $y$ also has a neighbour in $T_1$. Then $x,y \in \hullc (T_1)$, and hence $k=1$. Moreover, $\deg(x)=3$ and $\deg(y)\leq 3$. 

If $\deg(y)=1$, then $G\in \mathscr{F}_1$. If $\deg(y)=2$, then clearly $G\in \mathscr{F}_2$. Now consider the case $\deg(y)=3$. Suppose that $y$ has a neighbour $w\in T_1$ such that $w\notin \{u,v\}$. Then $\hullc (S\setminus \{u\}) \cup \hullc (S\setminus \{v\}) = \hullc (S)$, which contradicts the fact that $S$ is a $C$-independent set in $G$. Thus $u$ and $v$ are the only neighbours of $y$ in $S$, and hence $G\in \mathscr{F}_3$.

In the following, assume that $y$ has no neighbours in $T_1$. If $y$ has more than one neighbour in any $T_i$ ($2\leq i\leq k$), then $S$ becomes $C$-dependent. Since $\deg(y)\geq 2$, without loss of generality assume that $y$ has a unique neighbour in $T_2$, say $u'$. This implies that $x$ must have at least one neighbour in $T_2$, say $v'$. Again, since $S$ is $C$-independent, it follows that $k=2$, $\deg(x)=4$, and $\deg(y)=2$. As $G$ has no pendant vertices, $T_1$ is a path joining $u$ and $v$, and $T_2$ is a path joining $u'$ and $v'$. Therefore, $G\in \mathscr{F}_4$.

\medskip
\noindent \textbf{Case 2.} $xy \notin E(G)$.
    
Since $S$ is not convex in $G$, it follows that either $x$ or $y$ (say $x$) has at least two neighbors in some component $T_i$, say $T_1$. First, suppose that $k = 1$. Since $G[T_1 \cup \{x,y\}]$ contains no pendant edges, it follows that $T_1$ is a path. As $S$ is $C$-independent, either $x$ or $y$ has exactly two neighbors in $T_1$. Without loss of generality, assume that $x$ has exactly two neighbors in $T_1$. Then the induced subgraph $G[T_1 \cup \{x\}]$ forms a cycle. Consequently, $G \in \mathscr{F}_1$. Hence, we may assume that $k \geq 2$. As above, it follows that $G[T_1 \cup \{x\}]$ forms a cycle.

Now, suppose that $y$ has two neighbors in $T_1$. Then $x,y \in \hullc (T_1 )$. Since $k \geq 2$, this implies that $S$ is $C$-dependent, a contradiction. Hence, we may assume that $y$ has at most one neighbor in $T_1$. Using a similar argument, we deduce that $y$ has at most one neighbor in each component $T_i$, for all $2 \leq i \leq k$.
First, consider the case in which $y$ has a neighbor in $T_1$. Since $y \in \hullc (S)$, we may assume, without loss of generality, that $y$ has exactly one neighbor in another component $T_2$. Again, as $y \in \langle S \rangle$ and $S$ is $C$-independent, it follows that $k = 2$ and that $x$ has exactly one neighbor in $T_2$. Since $G[S \cup \{x,y\}]$ contains no pendant edges, it follows that $G \in \mathscr{F}_2$. Thus, in the follows, we assume that $y$ has no neighbors in $T_1$. As above,
since $y \in \hullc (S)$, the vertex $y$ must have at least one neighbor in at least two components of $G[S]$. Hence, without loss of generality, we may assume that $y$ has exactly one neighbor in each of $T_2$ and $T_3$. Again, since $y \in \hullc (S)$ and $S$ is $C$-independent, it follows that $x$ also has exactly one neighbor in each of $T_2$ and $T_3$. Consequently, we must have $k = 3$, and both $T_2$ and $T_3$ are paths. Since $G[S \cup \{x,y\}]$ contains no pendant edges, it follows that $G \in \mathscr{F}_4$.

Conversely, it is straightforward to verify that $\car(G) = n-2$ for all graphs $G$ belonging to the family $\displaystyle\bigcup_{i=1}^4 \mathscr{F}_i$. We therefore omit the proof. 
\end{proof}

\section{Carath\'eodory Number of Product Graphs}
\label{sec:products}

Let $G$ and $H$ be two graphs.
In this section we study the Carath\'eodory number of the \emph{Cartesian product} $G \Box H$, \emph{lexicographic product} $G \circ H$, and \emph{strong product} $G \boxtimes H$. The graphs $G$ and $H$ are called \textit{factors}. All these products share the vertex set $V(G) \times V(H)$. For $(g_1, h_1), (g_2, h_2) \in V(G) \times V(H)$:
In the Cartesian product $G \square H$, the vertices $(g_1, h_1)$ and $(g_2, h_2)$ are adjacent if and only if either 
i) $g_1 \sim g_2$ in $G$ and $h_1 = h_2$, or 
ii) $g_1 = g_2$ and $h_1 \sim h_2$ in $H$.
In the lexicographic product $G \circ H$, these vertices are adjacent if either 
i) $g_1 \sim g_2$, or 
ii) $g_1 = g_2$ and $h_1 \sim h_2$.
Finally, in the strong product $G \boxtimes H$, the vertices $(g_1, h_1)$ and $(g_2, h_2)$ are adjacent if one of the following holds: 
i) $g_1 \sim g_2$ and $h_1 = h_2$, 
ii) $g_1 = g_2$ and $h_1 \sim h_2$, or 
iii) $g_1 \sim g_2$ and $h_1 \sim h_2$.

If $\ast \in \{\square, \circ, \boxtimes\}$, then the \emph{projection mappings} $\pi_G: V(G \ast H) \rightarrow V(G)$ and $\pi_H: V(G \ast H) \rightarrow V(H)$ are given by $\pi_G(u,v) = u$ and $\pi_H(u,v) = v$, respectively. We adopt the following conventions. For $u \in V(G)$ and $v \in V(H)$, we define $^uH$ to be the subgraph of $G \ast H$ induced by $\{ u \} \times V(H)$, which we call an \emph{$H$-layer}, while the \emph{$G$-layer}, $G^v$ is the subgraph induced by $V(G) \times \{ v \}$. 
If $S \subseteq V(G \square H)$, then the set $\{g \in V(G) :\ (g,h) \in S \text{ for some } h \in V(H)\}$ is the \emph{projection} $\pi_G(S)$ of $S$ on $G$. The projection $\pi_H(S)$ of $S$ on $H$ is defined analogously. Finally, for $S \subseteq V(G \ast H)$, we call $S$ a \emph{sub-product} of $G \ast H$ if $S = \pi_G(S) \times \pi_H(S)$.

As observed in \cite[Remark~4.8]{anand2025caratheodory}, any two adjacent vertices of $G \ast H$, where $\ast \in \{\boxtimes,\circ\}$, form a hull set of $G \ast H$. Hence, every three-vertex subset of $G \ast H$ is $C$-dependent for $\ast \in \{\boxtimes,\circ\}$. Consequently, $\car(G \ast H)=2 \quad \text{for } \ast \in \{\boxtimes,\circ\}.$
Therefore, throughout the remainder of this section, we restrict our attention to the Carath\'eodory number of Cartesian product graphs. The following result on convex sets in Cartesian product graphs is required.
\begin{theorem}{\rm \label{theo:StimesT}\cite{anand2025complexity}}
Let $G$ and $H$ be two nontrivial connected graphs and let $S$ and $T$ be any two convex sets in $G$ and $H$ respectively. Then $S\times T$ is convex in $G\Box H$. 
\end{theorem}
In the following, we first establish a tight lower bound for the Carath\'eodory number of Cartesian product graphs, for which some preparation is needed.
\begin{lemma}\label{path_cycle}
    Let $G$ and $H$ be nontrivial connected graphs and let $S \subseteq V(G \square H)$. 
If $P$ is a path in the subgraph induced by $S$, then $\pi_G(P)\times \pi_H(P)\subseteq \langle S\rangle$.
\end{lemma}
\begin{proof}

 Consider a path 
$P : (x,y) = (g_1,h_1), (g_2,h_2), \ldots, (g_k,h_k) = (x',y')$
of length $k$ between $(x,y)$ and $(x',y')$ in the induced subgraph on $S$. We proceed by induction on $k$. If $k=1$ or $k=2$, the result holds trivially. 
Consider the case $k=3$. 
If $g_1=g_2=g_3$ or $h_1=h_2=h_3$, then $\pi_G(P)\times \pi_H(P)\subseteq \hullc (S)$.
Hence, without loss of generality, assume that $g_1=g_2$ and $g_2g_3\in E(G)$. 
Then $h_1h_2\in E(H)$ and $h_2=h_3$. 
Since $(g_1,h_1)$, $(g_2,h_2)$ and $(g_3,h_3)$ belong to $S$, and the vertices $(g_1,h_1),\ (g_2,h_2),\ (g_3,h_3),\ (g_3,h_1)$
form a cycle in $G \square H$, it follows that $(g_3,h_1)\in \hullc (S)$. Therefore,
$\pi_G(P)\times \pi_H(P)\subseteq \hullc (S)$.

Now assume that the result holds for all paths of length $i<k$. 
We prove it for a path $P : (x,y) = (g_1,h_1), (g_2,h_2), \ldots, (g_k,h_k) = (x',y')$ 
of length $k$ in the subgraph induced by $S$. Without loss of generality, we may assume that 
$g_1g_2 \in E(G)$ and $h_{k-1}h_k \in E(H)$. Consider the subpaths $P' : (x,y)=(g_1,h_1), (g_2,h_2), \ldots, (g_{k-1},h_{k-1})$ and $
P'' : (g_2,h_2), (g_3,h_3), \ldots, (g_k,h_k)=(x',y')$,
each of which is a path of length $k-1$ in the subgraph induced by $S$. 
By the induction hypothesis,
$\pi_G(P')\times \pi_H(P') \subseteq \hullc (S)$
and
$\pi_G(P'')\times \pi_H(P'') \subseteq \hullc (S)$.
Observe that
$(\pi_G(P)\times \pi_H(P)) 
\setminus 
((\pi_G(P')\times \pi_H(P')) 
\cup 
(\pi_G(P'')\times \pi_H(P''))) 
= \{(g_1,h_k), (g_k,h_1)\}$.
Now, 
$(g_1,h_{k-1}), (g_2,h_{k-1}), (g_2,h_k)
\in
(\pi_G(P')\times \pi_H(P'))
\cup
(\pi_G(P'')\times \pi_H(P''))
\subseteq
\hullc (S)$.\\
Since
$(g_1,h_{k-1}), (g_2,h_{k-1}), (g_2,h_k), (g_1,h_k)$
form a cycle in $G\Box H$, it follows that $(g_1,h_k)\in \hullc (S)$.

Similarly,
$(g_{k-1},h_1), (g_{k-1},h_2), (g_k,h_2)
\in
(\pi_G(P')\times \pi_H(P'))
\cup
(\pi_G(P'')\times \pi_H(P''))
\subseteq
\hullc (S)$, and since
$(g_{k-1},h_1), (g_{k-1},h_2), (g_k,h_2), (g_k,h_1)$
form a cycle in $G\Box H$, we obtain $(g_k,h_1)\in \hullc (S)$.

Therefore,
$\pi_G(P)\times \pi_H(P)\subseteq \hullc (S)$.
Hence, by mathematical induction, the result follows.
\end{proof}
\begin{lemma}\label{lem:subproduct} Let $G$ and $H$ be nontrivial connected graphs, and let $S$ be a connected convex set in $G \Box H$. Then the following statements hold:
\begin{enumerate}
    \item[$(1)$]  $S = \pi_G(S) \times \pi_H(S)$.
    \item[$(2)$] The projections $\pi_G(S)$ and $\pi_H(S)$ are convex sets in $G$ and $H$, respectively.
\end{enumerate}
\end{lemma}
\begin{proof} (1)
Let $g, g' \in \pi_G(S)$. Then there exist $h, h' \in V(H)$ such that $(g,h), (g',h') \in S$. 
To complete the proof, it suffices to show that $\{g,g'\} \times \{h,h'\} \subseteq S$. 
For this purpose, consider a path $P \colon (g,h) = (x_1,y_1), (x_2,y_2), \ldots, (x_k,y_k) = (g',h')$
in the induced subgraph of $S$.
We proceed by induction on $k$. 
For the base case $k = 1$, without loss of generality, the path $P$ is given by $(g,h), (g',h), (g',h')$. 
Since $S$ is convex, it follows that $(g,h') \in S$. 
Hence, the result holds for $k = 1$. Assume that the result holds for all paths of length \( i < k \). That is, if \( (x,y) \in S \) and there exists a path between \( (g,h) \) and \( (x,y) \) in the induced subgraph of \( S \) in \( G \Box H \) of length at most \( k \), then \( \{g,x\} \times \{h,y\} \subseteq S \).
Given this assumption to the  path $P$ implies that \( \{x_1, x_2, \ldots, x_{k-1}\} \times \{y_1, y_2, \ldots, y_{k-1}\} \subseteq S \). Now, since $(x_{k-1}, y_{k-1})$ and $(x_k, y_k)$ are adjacent in $G\Box H$, hence without loss of generality, we may assume that $y_k = y_{k-1}$. This in turn implies that $(x_k, y_{k-1}) \in S$. From the path $P$ in $G\Box H$, it is clear that \( \pi_H(P) : y_1, y_2, \ldots, y_{k-1} \) forms an $(h,h')$-walk in \( H \). Now, choose any \( y_i \) where \( 1 \leq i \leq k-2 \). Let \( Q \) be a \( (y_{k-1},y_i) \)-sub walk of \( \pi_H(P) \), say \( Q : y_{k-1}, v_1, v_2, \ldots, v_r = y_i \), where \( \{v_1, v_2, \ldots, v_r\} \subseteq \{y_1, y_2, \ldots, y_{k-1}\} \).
By the induction hypothesis, \( \{x_{k-1}\} \times \{v_1, v_2, \ldots, v_r\} \subseteq S \). This sequentially shows that \( (x_k, v_2), (x_k, v_3), \ldots, (x_k, v_r) \in S \). Therefore, \( \{x_k\} \times \{y_1, y_2, \ldots, y_{k-1}\} \subseteq S \).
This implies that $\{g,g'\} \times \{h,h'\} \subseteq S$.\\
To prove (2), suppose, to the contrary, that $\pi_G(S)$ is not convex. Then there exists a vertex $g \notin \pi_G(S)$ such that $g \in \hullc (\pi_G(S))$. Consequently, for any $h \in \pi_H(S)$, we have $(g,h) \in \hullc (\pi_G(S) \times \{h\} )$. This contradicts the statement~(1).
\end{proof}
Using the above two lemmas, we obtain the following lower bound.
\begin{theorem}\label{thm:lbound}
     Let $G$ and $H$ be two nontrivial connected graphs of orders $m$ and $n$, respectively. 
Let $P_r$ and $P_s$ be longest paths in $G$ and $H$, respectively, such that 
$\hullc (V(P_r)) = P_r$ and $\hullc (V(P_s)) = P_s$. 
Then $\car(G \Box H) \geq \max \{ r + s - 1,\; \car(G) + \car(H) - 1 \}$.
\end{theorem}
\begin{proof}
Let $P_r$ and $P_s$ be two longest paths in $G$ and $H$, respectively, such that 
$\hullc (V(P_r)) = V(P_r)$ and $\hullc (V(P_s)) = V(P_s)$. 
We fix 
$P_r : x_1,x_2,\ldots , x_r$ 
and 
$P_s : y_1,y_2,\ldots,y_s$.

\medskip
\noindent
{\bf Claim:} 
$S=\{(x_1,y_1), (x_2,y_1), \ldots, (x_r,y_1), (x_r,y_2),(x_r,y_3),\ldots, (x_r,y_s)\}$ 
is a $C$-independent set in $G\Box H$.

\medskip

It is clear that 
$(x_1,y_1), (x_2,y_1), \ldots, (x_r,y_1), (x_r,y_2),(x_r,y_3),\ldots, (x_r,y_s)$ 
forms a path in $G\Box H$. Then, by Lemma~\ref{path_cycle}, we obtain
$V(P_r)\times V(P_s)\subseteq \hullc (S)$.
On the other hand, recall that $V(P_r)$ and $V(P_s)$ are convex sets in $G$ and $H$, respectively. Consequently, by Theorem~\ref{theo:StimesT}, we get $V(P_r)\times V(P_s)$ is a convex set in $G\Box H$. Therefore, it follows that $\hullc (S) = V(P_r)\times V(P_s)$. Now consider the vertex $(x_1,y_s)\in \hullc (S)$. 
We claim that for any $(x_i,y_j)\in S\setminus \{(x_1,y_s)\}$, 
$(x_1,y_s)\notin \hullc (S\setminus \{(x_i,y_j)\})$. 
Fix $(x_i,y_j)\neq (x_1,y_s)$. First consider the case when $i=r$. 
If $j>1$, then $\hullc (S\setminus \{(x_i,y_j)\}) 
= V(P_r)\times \{y_1,y_2,\ldots, y_{j-1}\}
\;\cup\;
\{(x_r,y_{j+1}),(x_r,y_{j+2}),\ldots,(x_r,y_s)\}$.
On the other hand, if $j=1$, then
$\hullc (S\setminus \{(x_i,y_j)\})
=
\{(x_1,y_1),(x_2,y_1),\ldots, (x_{r-1},y_1)\}
\cup
\{(x_r,y_2),(x_r,y_{3}),\ldots,(x_r,y_s)\}$. Next, consider the case when $i< r$. Then $\hullc (S\setminus \{(x_i,y_j)\})
= \{(x_1,y_1), (x_2,y_1), \ldots, (x_{i-1},y_1)\}
\;\cup\;
(\{x_{i+1}, x_{i+2}, \ldots, x_r\}\times V(P_s))$.
Hence, $(x_1,y_s)\notin \hullc (S\setminus \{(x_i,y_j)\})$. 
This proves that $S$ is a $C$-independent set in $G\Box H$ with cardinality $r+s-1$.

\medskip

In the following, we prove that $\car(G\Box H)\geq \car(G) + \car(H) - 1$. 
Let $S=\{g_1,g_2,\ldots ,g_r\}$ and 
$T=\{h_1,h_2,\ldots,h_t\}$ 
be maximum $C$-independent sets in $G$ and $H$, respectively.
Choose $g\in \hullc (S)$ such that 
$g\notin\displaystyle \bigcup_{i=1}^r \hullc (S \setminus \{g_i\})$, and choose 
$h\in \hullc (T)$ such that 
$h\notin\displaystyle \bigcup_{j=1}^t \hullc (T \setminus \{h_j\})$.
We define the set $
U=(S\times\{h_1\})\, \cup \,(\{g_r\}\times T)$. Then both 
$\hullc (S \times \{h_1\}) \subseteq \hullc (U )$ and 
$\hullc (\{g_r\} \times T) \subseteq \hullc (U)$ hold. 
Since $\hullc (S \times \{h_1\}) \cup \hullc (\{g_r\} \times T )$ 
is connected in $G \Box H$, we can conclude that $\hullc (U)$ is also connected in $G \Box H$. Now, since $(g,h_1)$ and $(g_r,h)$ belong to $\hullc (U)$, it follows from Lemma~\ref{lem:subproduct} that $(g,h)$ also belongs to $\hullc (U)$.

\medskip
\noindent
{\bf Claim:} $(g,h)\notin \displaystyle\bigcup_{(g',h')\in U} \hullc (U\setminus \{(g',h')\})$.\\
 Let $(g',h') \in U$ be arbitrary. 
We first consider the case when $g' = g_i$ for some $1 \leq i < r$, in which case $h' = h_1$. 
By the given assumption, it follows that $
g \notin \hullc (S \setminus \{g'\}).
$
Let $C$ denote the connected component of the induced subgraph 
$\hullc (S \setminus \{g'\})$ that contains $g_r$. Then we have
\[
\hullc (U \setminus \{(g',h')\}) 
= \big( (\hullc (S \setminus \{g'\}) \setminus V(C)) \times \{h_1\} \big)
\cup \big( V(C) \times \hullc (T) \big).
\]
Consequently, for any $y \in V(H)$, it follows that 
$(g,y) \notin \hullc (U \setminus \{(g',h')\})$, and in particular,
$
(g,h) \notin \hullc (U \setminus \{(g',h')\}).
$
A similar argument applies to the case when $g' = g_r$ and $h' = h_j$ for some $1 < j \leq t$, which yields $
(g,h) \notin \hullc (U \setminus \{(g',h')\}).$

Next, we consider the case when $(g',h') = (g_r,h_1)$. In this situation, we obtain
\[\hullc (U \setminus \{(g',h')\}) \subseteq \hullc (S\setminus \{g_r\})\times \hullc (T\setminus \{h_1\})
\]
It follows once again that
$
(g,h) \notin \hullc (U \setminus \{(g',h')\}).
$
Therefore,
$
(g,h) \notin \displaystyle\bigcup_{(g',h') \in U} \hullc (U \setminus \{(g',h')\}).$
Hence, $U$ is a $C$-independent set in $G \Box H$. This completes the proof.
\end{proof}
Next, we determine the exact value of $\car(P_m \,\Box\, P_n)$ and this proves the sharpness of the lower bound in Theorem~\ref{thm:lbound}.
\begin{lemma}\label{lem:maxc-minh}
   For $m \geq n \geq 2$, every maximal $C$-independent set of $P_m \Box P_n$ is a minimal hull set of $P_m \Box P_n$.

\begin{proof}
We first consider the case $n=2$ and prove the result by induction on $m$. 
The statement is straightforward to verify for $m=2$. 
Assume, as the induction hypothesis, that every maximal $C$-independent set of $P_k \Box P_2$ is a minimal hull set of $P_k \Box P_2$. 
Let $S$ be a maximal $C$-independent set of $P_{k+1} \Box P_2$. 
Fix the paths $P_{k+1} : x_1, x_2, \ldots, x_{k+1}$ and $P_2 : y_1, y_2$, and let $G = P_{k+1} \Box P_2$. 
Suppose that $S \subseteq V(P_k \Box P_2)$. 
Since $P_k \Box P_2$ is a convex subgraph of $G$, it follows that $S$ is a maximal $C$-independent set of $P_k \Box P_2$. 
Hence, by the induction hypothesis, $S$ is a minimal hull set of $P_k \Box P_2$. 
Recall that the complement of each layer of $P_k \Box P_2$ is convex in $P_k \Box P_2$. 
This implies that $S$ intersects every layer of $P_k \Box P_2$. 
Without loss of generality, assume that $(x_1, y_1) \in S$. 
Moreover, either $(x_k, y_1) \in S$ or $(x_k, y_2) \in S$.

We first consider the case where $(x_k, y_1) \in S$. 
Let $j$ be the largest index such that $1 \leq j \leq k$ and $(x_j, y_2) \in S$. 
Since the subgraph induced by $S$ is connected, it follows from Lemma~\ref{path_cycle} that $
(x_k, y_2) \in \hullc (S \setminus \{(x_k, y_2)\}).
$
Therefore, $(x_k, y_2) \notin S$, since $S$ is $C$-independent. 
Consequently, we must have $j \leq k-1$.
We now show that $j=1$. 
Suppose, for a contradiction, that $j>1$. 
Since the subgraph induced by $S$ is connected in $P_k \Box P_2$, there exist three distinct paths $Q_1, Q_2,$ and $Q_3$ joining the vertices $(x_1, y_1)$, $(x_j, y_2)$, and $(x_k, y_1)$ pairwise. 
Consequently, Lemma~\ref{path_cycle} implies that $S$ is $C$-dependent, contradicting the assumption that $S$ is $C$-independent. 
Hence, $j=1$. Thus, as the subgraph induced by $S$ is connected in $P_k \Box P_2$, we must have
$S=\{(x_1,y_2),(x_1,y_1),(x_2,y_1),\ldots,(x_k,y_1)\}.$
This further implies that the set $S \cup \{(x_{k+1}, y_1)\}$ is $C$-independent in $P_{k+1} \Box P_2$, contradicting the maximality of $S$.

Thus, $S$ must intersect $\{x_{k+1}\}\times V(P_2)$ and, similarly, $\{x_1\}\times V(P_2)$. 
Since $S$ induces a connected subgraph, it follows that $S \cap (\{x_i\}\times V(P_2)) \neq \emptyset$ for all $i \in [k+1]$. This in turn implies that $S$ is a hull set of $P_{k+1} \Box P_2$. 
Moreover, since $S$ is $C$-independent, it follows that $S$ is a minimal hull set of $P_{k+1} \Box P_2$.

Next, consider the case where $(x_k, y_2) \in S$. 
Since the subgraph induced by $S$ is connected in $P_k \Box P_2$, it follows that either $(x_1, y_2) \in S$ or $(x_k, y_1) \in S$. If $(x_1, y_2) \in S$, then, by an argument analogous to the previous case, we obtain $
S=\{(x_1,y_1),(x_2,y_1),(x_2,y_2),\ldots,(x_k,y_2)\}.$
In this case, $S \cup \{(x_{k+1}, y_2)\}$ forms a strictly larger $C$-independent set, contradicting the maximality of $S$.
Hence, $(x_k, y_1) \in S$. 
Thus, $S=\{(x_1,y_1),(x_2,y_1),\ldots,(x_k,y_1),(x_k,y_2)\}.$
Now consider the set
$S'=\{(x_1,y_1),(x_1,y_2),(x_2,y_2),\ldots,(x_k,y_2)\}.$
Then $|S'|=|S|$, and $S'$ is also a $C$-independent set of $P_k \Box P_2$, contradicting the conclusion established above. 

Thus, as in the first case, $S$ must intersect every layer of $P_{k+1} \Box P_2$, which implies that $S$ is a hull set of $P_{k+1} \Box P_2$. 
Since $S$ is $C$-independent, it follows that $S$ is a minimal hull set of $P_{k+1} \Box P_2$. 

Hence, by induction, every maximal $C$-independent set of $P_n \Box P_2$ is a minimal hull set of $P_n \Box P_2$ for all $n \geq 2$.

We now consider the case $n \geq 3$. We show that every maximal $C$-independent set of $P_m \Box P_n$ is a minimal hull set of $P_m \Box P_n$ for all $m \geq n \geq 3$.
Assume, for a contradiction, that the result fails. 
Let $n$ be the minimum integer such that there exists a maximal $C$-independent set $S$ of $P_m \Box P_n$ which is not a hull set. 
Then necessarily $n \geq 3$. Fix the paths $P_m : x_1, x_2, \ldots, x_m$ and $P_n : y_1, y_2, \ldots, y_n$, and let $G = P_m \Box P_n$. 
Since the subgraph induced by $S$ is connected (by Lemma~\ref{path_cycle}), we may assume, without loss of generality, that $S \subseteq V(P_m \Box P_{n-1}).$
Since $P_m \Box P_{n-1}$ is a convex subgraph of $G$, it follows that $S$ is a maximal $C$-independent set of $P_m \Box P_{n-1}$. 
By the minimal choice of $n$, it follows that $S$ is a minimal hull set of $P_m \Box P_{n-1}$. 
In particular, $S$ intersects every layer of $P_m \Box P_{n-1}$.
Let $i$ be the smallest index such that $(x_i, y_1) \in S$, $j$ the largest index such that $(x_j, y_n) \in S$, $r$ the smallest index such that $(x_1, y_r) \in S$, and $s$ the largest index such that $(x_m, y_s) \in S$.
We first show that $i = r = 1$. 
Suppose, to the contrary, that $i \neq r$. 
Let $P_1$ be a $(x_i, y_1),(x_1, y_r)$-path, $P_2$ a $(x_i, y_1),(x_m, y_s)$-path, $P_3$ a $(x_1, y_r),(x_j, y_n)$-path, and $P_4$ a $(x_j, y_n),(x_m, y_s)$-path in the subgraph induced by $S$.
Then, by Lemma~\ref{path_cycle},
$\displaystyle\bigcup_{c=1}^4 \hullc (V(P_c)) = \hullc (S),$
which implies that $S$ is $C$-dependent, a contradiction. 
Hence, $i = r = 1$. 
Proceeding as above, we obtain $j = m$. Therefore, the subgraph induced by $S$ in $P_m \Box P_{n-1}$ is a path joining $(x_1, y_1)$ and $(x_m, y_{n-1})$. However, in this case, $S \cup \{(x_m, y_n)\}$ is a $C$-independent set of $P_m \Box P_n$, contradicting the maximality of $S$.
Thus, every maximal $C$-independent set of $P_m \Box P_n$ is a minimal hull set of $P_m \Box P_n$. 
This completes the proof.
 \end{proof}
\end{lemma}

\begin{theorem}
   For $m \geq n \geq 2$, we have $\car(P_m \Box P_n) = m + n - 1$.
    \begin{proof}
       Fix the paths $P_m : x_1, x_2, \ldots, x_m$ and $P_n : y_1, y_2, \ldots, y_n$, and consider the set
$T=\{(x_1,y_1),(x_2,y_1),\ldots,(x_m,y_1)\}
\cup
\{(x_1,y_1),(x_1,y_2),\ldots,(x_1,y_n)\}.
$
Then $T$ is a $C$-independent set of size $m+n-1$. 
Hence, $\car(G)\geq m+n-1$.

Conversely, let $S$ be a maximum $C$-independent set of $P_m \Box P_n$. 
Then $S$ is maximal, and by Lemma~\ref{lem:maxc-minh}, it forms a minimal hull set of $P_m \Box P_n$. 
Hence, $S$ intersects every layer of $P_m\Box P_n$, 
since the complement of each layer in $P_m \Box P_n$ is convex.
Since $S$ is $C$-independent, it follows, as in the proof of the preceding lemma, 
that the subgraph of $P_m \Box P_n$ induced by $S$ is a path $P$. 
Without loss of generality, we may assume that $S$ induces a path from $(x_1,y_1)$ to $(x_m,y_n)$ in $P_m \Box P_n$.
Since $S$ is $C$-independent, Lemma~\ref{path_cycle} ensures that $
S \;\cap\; \bigl(V(P_m \setminus \{x_1,x_m\}) \times V(P_n \setminus \{y_1,y_n\})\bigr) = \varnothing.$
These observations together yield that $S$ must be the union of a single $x$-layer and a single $y$-layer. 
Hence, $|S| \leq m + n - 1.$
 \end{proof}
\end{theorem}

\section{Final Considerations}
\label{sec:conclusions}

In this paper, we have investigated the Carath\'{e}odory number in cycle convexity. Given a graph $G$ and a positive integer $k$, we have shown that deciding whether $\car(G) \geq k$ is \NP-complete even for bipartite graphs. This motivates further investigation of parameterized and approximation aspects of the problem.

On the other hand, we have established constant upper bounds and exact values for several graph classes, including forests, cycles, complete, complete multipartite, split, and $P_4$-sparse graphs. We also characterized graphs $G$ with values close to the maximum, namely those with $\car(G)=n-1$ and $\car(G)=n-2$, where $n = |V(G)|$. As a direction for future work, it would be natural to extend these results to broader graph classes such as chordal graphs (which generalize split graphs) and other classes with few $P_4$’s. The $P_4$-tidy in particular, a superclass of $P_4$-sparse graphs, may contain induced $C_5$, then it would be interesting to determine whether a larger constant upper bound holds in this setting, possibly contrasting with the bound of $3$ for $P_4$-sparse graphs.

In addition, we studied the behavior of the Carath\'{e}odory number under graph products, obtaining bounds and exact values for the strong, lexicographic, and Cartesian products. 
It would be interesting to investigate the parameter under other graph products, such as the direct product, as well as to establish equality results for other classes of factors, extending the case of products of paths considered in this paper.

\section*{Acknowledgements}

Revathy S.  Nair acknowledges the financial support from the University of Kerala, for providing the University Junior Research Fellowship (Ac EVI 3217/2025/UOK dated 10/04/2025).\\

\bibliographystyle{amsplain}
\bibliography{cycle}

@article{coelho2019p3,
title = {{On the $P_3$-hull number of some products of graphs}},
journal = {Discrete Applied Mathematics},
volume = {253},
pages = {2--13},
year = {2019},
issn = {0166-218X},
doi = {https://doi.org/10.1016/j.dam.2018.04.024},
url = {https://www.sciencedirect.com/science/article/pii/S0166218X18302440},
author = {Erika M. M. Coelho and Hebert Coelho and Julliano R. Nascimento and Jayme L. Szwarcfiter}
}

@article{dreyer2009irreversible,
  title={{Irreversible $k$-threshold processes: Graph-theoretical threshold models of the spread of disease and of opinion}},
  author={Dreyer Jr, Paul A and Roberts, Fred S},
  journal={Discrete Applied Mathematics},
  volume={157},
  number={7},
  pages={1615--1627},
  year={2009},
  publisher={Elsevier}
}

@article{interval08,
  title={On interval number in cycle convexity},
  author={Araujo, Julio and Ducoffe, Guillaume and Nisse, Nicolas and Suchan, Karol},
  journal={Discrete Mathematics \& Theoretical Computer Science},
  volume={20},
  number={Graph Theory},
  year={2018},
  publisher={Episciences. org}
}

@article{hull09,
  title={On the hull number on cycle convexity of graphs},
  author={Araujo, Julio and Campos, Victor and Gir{\~a}o, Darlan and Nogueira, Jo{\~a}o and Salgueiro, Ant{\'o}nio and Silva, Ana},
  journal={Information Processing Letters},
  volume={183},
  pages={106420},
  year={2024},
  publisher={Elsevier}
}

@book{van22,
  title={Theory of convex structures},
  author={van De Vel, Marcel L.J.},
  year={1993},
  publisher={Elsevier}
}

@book{araujo2025introduction,
  title={Introduction to Graph Convexity: An Algorithmic Approach},
  author={Ara{\'u}jo, J{\'u}lio and Dourado, Mitre C and Protti, F{\'a}bio and Sampaio, Rudini M},
  year={2025},
  publisher={Springer Nature}
}

@article{bijo3,
  title={{On the $\Delta$-interval and the $\Delta$-convexity numbers of graphs and graph products}},
  author={Anand, Bijo S and Dourado, Mitre C and Narasimha-Shenoi, Prasanth G and Ramla, Sabeer S},
  journal={Discrete Applied Mathematics},
  volume={319},
  pages={487--498},
  year={2022},
  publisher={Elsevier}
}

@article{bijo2,
  title={{Computing the hull number in $\Delta$-convexity}},
  author={Anand, Bijo S and Anil, Arun and Changat, Manoj and Dourado, Mitre C and Ramla, Sabeer S},
  journal={Theoretical Computer Science},
  volume={844},
  pages={217--226},
  year={2020},
  publisher={Elsevier}
}

@inproceedings{bijo1,
  title={{$\Delta$-Convexity Number and $\Delta$-Number of Graphs and Graph Products}},
  author={Anand, Bijo S and Narasimha-Shenoi, Prasanth G and Ramla, Sabeer Sain},
  booktitle={Conference on Algorithms and Discrete Applied Mathematics},
  pages={209--218},
  year={2020},
  organization={Springer}
}

@article{source9,
  title={Steiner distance and convexity in graphs},
  author={C{\'a}ceres, Jos{\'e} and M{\'a}rquez, Alberto and Puertas, Mar{\'\i}a Luz},
  journal={European Journal of Combinatorics},
  volume={29},
  number={3},
  pages={726--736},
  year={2008},
  publisher={Elsevier}
}

@article{source11,
  title={{Graphs with few $P_4$’s under the convexity of paths of order three}},
  author={Campos, Victor and Sampaio, Rudini M and Silva, Ana and Szwarcfiter, Jayme L},
  journal={Discrete Applied Mathematics},
  volume={192},
  pages={28--39},
  year={2015},
  publisher={Elsevier}
}

@article{source17,
  title={Complexity results related to monophonic convexity},
  author={Dourado, Mitre C and Protti, F{\'a}bio and Szwarcfiter, Jayme L},
  journal={Discrete Applied Mathematics},
  volume={158},
  number={12},
  pages={1268--1274},
  year={2010},
  publisher={Elsevier}
}

@article{source20,
  title={{Convex sets in graphs, II. Minimal path convexity}},
  author={Duchet, Pierre},
  journal={Journal of Combinatorial Theory, Series B},
  volume={44},
  number={3},
  pages={307--316},
  year={1988},
  publisher={Elsevier}
}

@inproceedings{duchet1987convexity,
  title={Convexity in combinatorial structures},
  author={Duchet, Pierre},
  booktitle={Proceedings of the 14th Winter School on Abstract Analysis},
  pages={261--293},
  year={1987},
  organization={Circolo Matematico di Palermo}
}

@techreport{araujo2025on,
  TITLE = {{On the rank and the general position number in cycle convexity}},
  AUTHOR = {Ar{\'a}ujo, J{\'u}lio and Ar{\'a}ujo, Samuel N and Medeiros, Pedro P and Nisse, Nicolas and Silva, Caroline},
  URL = {https://inria.hal.science/hal-05378038},
  INSTITUTION = {{Inria \& Universit{\'e} Cote d'Azur, CNRS, I3S, Sophia Antipolis, France}},
  YEAR = {2025},
  HAL_ID = {hal-05378038},
  HAL_VERSION = {v1},
}

@article{gomes2025some,
  title={Some complexity results on cycle-convex partitions},
  author={Gomes, Guilherme CM and Lopes, Laila MV and dos Santos, Vinicius F},
  journal={Procedia Computer Science},
  volume={273},
  pages={365--372},
  year={2025},
  publisher={Elsevier}
}

@article{dourado2013caratheodory,
  title={On the Carath{\'e}odory number of interval and graph convexities},
  author={Dourado, Mitre C and Rautenbach, Dieter and Dos Santos, Vin{\'\i}cius Fernandes and Sch{\"a}fer, Philipp M and Szwarcfiter, Jayme L},
  journal={Theoretical Computer Science},
  volume={510},
  pages={127--135},
  year={2013},
  publisher={Elsevier}
}

@article{coelho2014caratheodory,
  title={{The Carath{\'e}odory number of the $P_3$ convexity of chordal graphs}},
  author={Coelho, Erika MM and Dourado, Mitre C and Rautenbach, Dieter and Szwarcfiter, Jayme L},
  journal={Discrete Applied Mathematics},
  volume={172},
  pages={104--108},
  year={2014},
  publisher={Elsevier}
}

@phdthesis{hoang1985perfect,
  author  = {C. T. Ho{\`a}ng},
  title   = {Perfect graphs},
  school  = {School of Computer Science, McGill University},
  address = {Montreal},
  year    = {1985}
}

@article{jamison1992recognizing,
  title={{Recognizing $P_4$-sparse graphs in linear time}},
  author={Jamison, Beverly and Olariu, Stephan},
  journal={SIAM Journal on Computing},
  volume={21},
  number={2},
  pages={381--406},
  year={1992},
  publisher={SIAM}
}

@inproceedings{caceres-2005,
  title={On monophonic sets in graphs},
  author={C{\'a}ceres, J and Hernando, C and Mora, M and Puertas, M and Seara, C},
  booktitle={20th British Combinatorial Conference, Durham-England},
  year={2005}
}

@article{centeno2011irreversible,
  title={Irreversible conversion of graphs},
  author={Centeno, Carmen C and Dourado, Mitre C and Penso, Lucia Draque and Rautenbach, Dieter and Szwarcfiter, Jayme L},
  journal={Theoretical Computer Science},
  volume={412},
  number={29},
  pages={3693--3700},
  year={2011},
  publisher={Elsevier}
}

@book{buckley-1990,
  title={Distances in Graphs},
  author={F Buckley, F Harary},
  year={1990},
  publisher={Eddison Wesley, Redwood City CA}
}

@article{farber-1986,
  title={Convexity in graphs and hypergraphs},
  author={Farber, Martin and Jamison, Robert E},
  journal={SIAM Journal on Algebraic Discrete Methods},
  volume={7},
  number={3},
  pages={433--444},
  year={1986},
  publisher={SIAM}
}

@article{everett1985hull,
  title={The hull number of a graph},
  author={Everett, Martin G and Seidman, Stephen B},
  journal={Discrete Mathematics},
  volume={57},
  number={3},
  pages={217--223},
  year={1985},
  publisher={Elsevier}
}

@article{araujo2020cycle,
  title={Cycle convexity and the tunnel number of links},
  author={Ara{\'u}jo, J{\'u}lio and Campos, Victor and Gir{\~a}o, Darlan and Nogueira, Jo{\~a}o and Salgueiro, Ant{\'o}nio and Silva, Ana},
  journal={arXiv preprint arXiv:2012.05656},
  year={2020},
  pages={1--28}
}

@article{lima2024complexity,
      title={The complexity of convexity number and percolation time in the cycle convexity}, 
      author={Carlos V G C Lima and Thiago Marcilon and Pedro P Medeiros},
      journal={arXiv preprint arXiv:2404.09236},
      year={2024},
      pages={1--24}
}

@book{garey1979computers,
  title={Computers and intractability},
  author={Garey, Michael R and Johnson, David S},
  volume={174},
  year={1979},
  publisher={Freeman San Francisco}
}

@article{barbosa2012caratheodory,
  title={On the Carath{\'e}odory number for the convexity of paths of order three},
  author={Barbosa, Rommel M and Coelho, Erika MM and Dourado, Mitre C and Rautenbach, Dieter and Szwarcfiter, Jayme L},
  journal={SIAM Journal on Discrete Mathematics},
  volume={26},
  number={3},
  pages={929--939},
  year={2012},
  publisher={SIAM}
}

@article{anand2025complexity,
  title={Complexity and structural results for the hull and convexity numbers in cycle convexity for graph products},
  author={Anand, Bijo S and SV, Ullas Chandran and Nascimento, Julliano R and Nair, Revathy S},
  journal={Discrete Applied Mathematics},
  volume={377},
  pages={552--561},
  year={2025},
  publisher={Elsevier}
}

@article{anand2025caratheodory,
  title={{Carath{\'e}odory number and exchange number in $\Delta$-convexity}},
  author={Anand, Bijo S and Anil, Arun and Changat, Manoj and Narasimha-Shenoi, Prasanth G and Ramla, Sabeer S},
  journal={J. Combin. Math. Combin. Comput},
  volume={126},
  number={11},
  pages={27},
  year={2025}
}

\end{document}